\documentclass[a4paper,10pt]{amsart}

\usepackage[english]{babel}
\usepackage[utf8]{inputenc}

\usepackage{mathrsfs}
\usepackage{mathtools}
\usepackage{dsfont}
\usepackage{amssymb}
\usepackage{amsmath}
\usepackage{enumerate}
\usepackage{amsfonts}
\usepackage{amsthm,amsmath}
\usepackage{color}
\usepackage{hyperref}

\numberwithin{equation}{section}
\newcommand{\bbC}{\mathds{C}}
\newcommand{\N}{\mathds{N}}
\newcommand{\bbN}{\mathds{N}}

\newcommand{\bbR}{\mathds{R}}

\newcommand{\cP}{\mathscr{P}}

\newcommand{\cL}{\mathscr{L}}
\newcommand{\cU}{\mathscr{U}}
\newcommand{\cT}{\mathscr{T}}
\newcommand{\cS}{\mathscr{S}}

\newcommand{\calL}{\mathscr{L}}

\newcommand{\applied}[2]{\langle #1,#2\rangle}
\DeclarePairedDelimiter\norm{\lVert}{\rVert}
\DeclarePairedDelimiter\abs{\lvert}{\rvert}
\newcommand{\argument}{\,\cdot\,}
\renewcommand{\phi}{\varphi}

\newcommand{\dx}{\;\mathrm{d}}
\newcommand{\eps}{\varepsilon}

\DeclareMathOperator{\id}{id}

\definecolor{mygreen}{rgb}{0.1,0.75,0.2}

\theoremstyle{definition}
\newtheorem{definition}{Definition}[section]

\newtheorem{remarks}[definition]{Remarks}
\newtheorem{example}[definition]{Example}

\theoremstyle{plain}
\newtheorem{proposition}[definition]{Proposition}
\newtheorem{lemma}[definition]{Lemma}
\newtheorem{theorem}[definition]{Theorem}
\newtheorem{corollary}[definition]{Corollary}

\begin{document}

\title[Lower Bounds and Asymptotics of Semigroups]{Lower Bounds and the Asymptotic Behaviour of Positive Operator Semigroups}
\author{Moritz Gerlach}
\email{moritz.gerlach@uni-potsdam.de}
\address{Moritz Gerlach\\Universit\"at Potsdam\\Institut f\"ur Mathematik\\Karl-Liebknecht-Stra{\ss}e 24--25\\14476 Potsdam\\Germany}
\author{Jochen Gl\"uck}
\email{jochen.glueck@uni-ulm.de}
\address{Jochen Gl\"uck, Institute of Applied Analysis, Ulm University, Germany}
\date{\today}
\begin{abstract}
	If $(T_t)$ is a semigroup of Markov operators on an $L^1$-space that admits a non-trivial lower bound, 
	then a well-known theorem of Lasota and Yorke asserts that the semigroup is strongly convergent as $t \to \infty$. 
	In this article we generalise and improve this result in several respects. 

	First, we give a new and very simple proof for the fact that the same conclusion also holds if the semigroup is merely assumed to be bounded instead of Markov.
	As a main result we then prove a version of this theorem for semigroups which only admit certain individual lower bounds. 
	Moreover, we generalise a theorem of Ding on semigroups of Frobenius-Perron operators.
	We also demonstrate how our results can be adapted to the setting of general Banach lattices and we give some counterexamples to show optimality of our results.
	
	Our methods combine some rather concrete estimates and approximation arguments with abstract functional analytical tools. 
	One of these tools is a theorem which relates the convergence of a time-continuous operator semigroup to the convergence of embedded discrete semigroups.
\end{abstract}
\maketitle

\section{Introduction and Preliminaries} \label{section:introduction-and-preliminaries}

This article is about the long-term behaviour of semigroups of operators. 
By a \emph{semigroup} $\cT$ we mean a family of bounded linear operators $(T_t)_{t \in J}$ on a Banach space $E$ 
where the index set $J$ is either $\bbN$ or $(0,\infty)$ and where the semigroup law $T_{t+s} = T_t T_s$ is fulfilled for all $s,t \in J$.
Operator semigroups describe the evolution over time of linear autonomous systems and thus occur in various contexts.
Of particular importance in applications are positive semigroups;
if $E$ is a Banach lattice, then a semigroup $\cT = (T_t)_{t \in J}$ is called \emph{positive} if each $T_t$ is a positive operator,
i.e.\ if $T_tf \ge 0$ for all $0 \le f \in E$ and all $t \in J$. 

There are many different methods to study the long-term behaviour of such semigroups. For example, the asymptotics of positive $C_0$-semigroups can be studied effectively by means of spectral theory
since the generators of positive semigroups exhibit very special spectral properties; we refer to the classical monograph \cite{nagel1986} for an overview over this field and to a series 
of papers \cite{davies2005, keicher2006, arendt2008, wolff2008, gerlach2013b} for several recent results. 
We also point out that some of these methods can be extended to study $C_0$-semigroups which fulfil weaker or different geometric properties than positivity, see e.g.~\cite{daners2016} and \cite{gluck2016}.

Spectral theory also comes in handy in the study of time-discrete positive semigroups since there is a well-developed spectral theory for positive operators on Banach lattices. 
For an overview of important parts of this so-called \emph{Perron-Frobenius theory} we refer to the survey article~\cite{grobler1995}; some very recent results can for example be found in \cite{gluck2015}.

On the other hand, there are also methods to analyse the asymptotic behaviour of positive semigroups which do not employ spectral theory. Many of those were first developed on $L^1$-spaces and 
have then been extended to more general Banach lattices. We refer to the monograph \cite{emelyanov2007} and the survey article \cite{emelyanov2006} for an excellent overview of such results. 

Of course, $L^1$-spaces have always been of particular importance in applications for several reasons. 
For instance, the theory of Markov chains and Markov processes naturally leads to so-called Markov semigroups on $L^1$-spaces, i.e.\ to positive semigroups which operate isometrically on the positive cone. 
Rather simple convergence criteria for such Markov semigroups are available in the finite dimensional case (see e.g.~\cite[Thm~VI.1.1]{doob1953} and \cite[Thm~4.2]{seneta1981}) while such criteria are more 
delicate in the infinite dimensional case, and in particular on non-discrete state spaces (see e.g.~\cite[Sec~4.2]{daprato1996} and \cite[Thm~4.4]{gerlach2015} for different versions 
and proofs of a classical theorem of Doob which addresses this issue; see also \cite[Thm~3.6]{gerlach2014} for a related Tauberian theorem). 

Positive semigroups on $L^1$-spaces also occur in ergodic theory. Here one studies the asymptotic behaviour of semigroups of so-called \emph{Frobenius-Perron operators} 
which are associated to dynamical systems on measure spaces, see e.g.\ \cite{lasota1982, ding2003} and the monograph \cite{lasota1994}; see also \cite{nagel1991} and the recent monograph \cite{eisner2015} for some related topics.
Other applications of positive semigroups on $L^1$-space include models from mathematical biology (see \cite[Sec~5]{pichor2012} for some examples), 
the Boltzmann equation (see e.g.~\cite[Sec~8--9]{lasota1983}) and transport equations (see e.g.~\cite{pichor2000}).

This wide range of applications has motivated the development of many sufficient conditions for strong convergence of positive semigroups on $L^1$-spaces. 
The survey papers \cite{rudnicki2002, bartoszek2008, pichor2012} give an excellent overview of many such theorems. We also refer to the somewhat older paper \cite{komornik1993} which contains a wealth of interesting results.
Yet, even on these spaces the theory still seems to be far from being complete.

\subsection*{Contribution of this article}

One classical method to obtain semigroup convergence on $L^1$-spaces is the so-called \emph{lower bound} technique. It was developed by Lasota and Yorke (\cite{lasota1982}; 
see also \cite{lasota1983}) who proved that a semigroup of Markov operators converges automatically to a rank-1 projection if it admits a so-called non-trivial lower bound; 
see Section~\ref{section:lasota-yorke} for details. Later on, more general convergence results were discovered which yield this result as a corollary; see again Section~\ref{section:lasota-yorke}, 
in particular the comments before Corollary~\ref{cor:lasota-yorke-for-bounded-semigroups}, for a detailed discussion and several references. 
Yet, it appears that there is still some unexplored potential in classical approaches based on the existence of lower bounds
and we are going to prove several results which, to the best of our knowledge, 
cannot simply be derived from known convergence theorems. One of our main results asserts that in the Lasota-Yorke theorem it suffices to assume the existence of an
\emph{individual lower bound} for each orbit. The limit operator, however, does then no longer need to be of rank $1$. The complete statement reads as follows.

\begin{theorem} \label{thm:main-result-ind-lower-bounds}
	Let $(\Omega,\mu)$ be an arbitrary measure space and let $\cT = (T_t)_{t \in J}$ be a bounded positive semigroup on $L^1 \coloneqq L^1(\Omega,\mu)$ where either $J = \bbN$ or $J = (0,\infty)$. Suppose that for every 
	normalised function $0 \le f \in L^1$ there exists a function $0 \le h_f \in L^1$ such that $\norm{(T_tf - h_f)^-} \to 0$ as $t \to \infty$ 
	and such that $\inf_f \norm{h_f} > 0$. 
	
	Then $(T_t)$ converges strongly as $t \to \infty$.
\end{theorem}

This theorem follows from Corollary~\ref{cor:ind-lower-bounds-with-norm-bounds-for-bounded-sg} below. It is related to a paper of Ding \cite{ding2003} who 
considered the individual lower bounds $h_f$ in the above theorem under somewhat different assumptions; see Sections~\ref{section:individual-lower-bounds} and~\ref{section:individual-lower-bounds-without-a-norm-bound} for a detailed discussion. 
Note that the limit operator in Theorem~\ref{thm:main-result-ind-lower-bounds} can have infinite dimensional range and that, e.g.,\ the semigroup which consists only of the 
identity operator fulfils the assumptions of the theorem. This shows that Theorem \ref{thm:main-result-ind-lower-bounds} 
cannot be a special case of any result which contains some kind of compactness condition.

In order to give a complete summary of all our results, let us briefly outline the content of the paper:
In Section~\ref{section:three-abstract-convergence-theorems} we prove several abstract theorems about the convergence of semigroups on Banach spaces.
We show that, under appropriate boundedness assumptions, a semigroup $\cT = (T_t)_{t \in (0,\infty)}$ on a Banach space is convergent if 
all the embedded discrete semigroups $(T_{tn})_{n \in \bbN}$ are convergent (Theorem~\ref{thm:discrete-to-continuous}).
This result is implicitly contained in some proofs in the literature but we could not find an explicit statement of it.
In case that the limit operator has rank-$1$ the result can be further improved (Theorem~\ref{thm:discrete-to-continuous-rank-1}).
Moreover, we demonstrate how operator norm convergence of a semigroup can be analysed by considering strong convergence of a lifting of 
the semigroup to an ultra power (Theorem~\ref{thm:strong-ultra-convergence-implies-operator-norm-convergence}).

In Section~\ref{section:lasota-yorke} we revisit the classical Lasota-Yorke theorem about convergence of Markov semigroups which admit a non-zero lower bound. 
It is known that this result also holds for positive semigroups that are merely bounded and we give a new and very simple proof for this fact 
(see Corollary~\ref{cor:lasota-yorke-for-bounded-semigroups}). 
Then we discuss different ways to obtain a version of the Lasota-Yorke theorem for convergence in operator norm (see Corollary~\ref{cor:lasota-yorke-for-bounded-semigroups-and-convergence-in-operator-norm} 
and the subsequent discussion) and finally, we show that lower bounds in the sense of Lasota and Yorke cannot exist on any Banach lattices but on AL-spaces (Theorem~\ref{thm:lower-bounds-only-on-al-spaces}).

In Section~\ref{section:individual-lower-bounds} we prove a (slightly sharpened version of) the already stated Theorem~\ref{thm:main-result-ind-lower-bounds} (Corollary~\ref{cor:ind-lower-bounds-with-norm-bounds-for-bounded-sg})
and from this we derive a convergence result about dominating semigroups (Corollary~\ref{cor:domination}).

In Section~\ref{section:individual-lower-bounds-without-a-norm-bound} we show that Theorem~\ref{thm:main-result-ind-lower-bounds} also holds 
without the assumption $\inf_{f} \norm{h_f} > 0$ if the adjoint of each operator $T_t$ is a lattice homomorphism (Theorem~\ref{thm:ding-abstract-and-for-bounded-sg}); this generalises a result of Ding (see the discussion after Theorem~\ref{thm:ding-abstract-and-for-bounded-sg}). Moreover, we briefly consider lower bounds for semigroups of lattice homomorphisms (Theorem~\ref{thm:ind-lower-bounds-for-sg-of-lattice-homomorphisms}) and we prove a negative result about operator norm convergence of semigroups of Frobenius-Perron operators (Corollary~\ref{cor:norm-convergence-frobenius-perron-semigroups}).

In the final Section~\ref{section:lower-bounds-general-banach-lattices} we show several Lasota-Yorke type results on more general Banach lattices, where the notion of lower bounds has to be adjusted appropriately (see Theorems~\ref{thm:ind-lower-bound-wrt-psi} and~\ref{thm:unif-lower-bound-wrt-psi} as well as Propositions~\ref{prop:no-fixed-point-but-compact-orbits} and~\ref{prop:no-fixed-point-but-compact-orbits-without-os-norm}).

\subsection*{Preliminaries}

Throughout this paper, all Banach spaces are real unless stated otherwise. In all proofs we assume tacitly that the underlying Banach space $E$ is non-zero; it is, however, easy to see that our results themselves are also valid if $E = \{0\}$. If $E$ is a real or complex Banach space, then the space of 
all bounded linear operators on $E$ is denoted by $\calL(E)$. The \emph{dual space} of $E$ is denoted by $E'$ and for every $T \in \calL(E)$ we denote by $T' \in \calL(E')$ the \emph{adjoint} of $T$. For every $f \in E$ and every $\varphi \in E'$ the operator $\varphi \otimes f \in \calL(E)$ is defined by $(\varphi \otimes f)g = \langle \varphi, g \rangle f$ for all $g \in E$.
The identity operator on a Banach space $E$ will be denoted by $\id_E$.

We assume the reader to be familiar with the basic 
theory of Banach lattices; standard references for this topic are e.g.\ \cite{schaefer1974} and \cite{meyer1991}. 
Since there exist some different conventions concerning the notation, we summarise the most import notions in the following.
Let $E$ be a Banach lattice. Then $E_+ \coloneqq \{f \in E: f \ge 0\}$ denotes the \emph{positive cone} in $E$; a vector $f \in E$ is called \emph{positive} if $f \ge 0$. 
For two vectors $f,g \in E$ we write $f < g$ to indicate that $f \le g$ but $f \not= g$; in particular, $f > 0$ means that $f \ge 0$ but $f \not= 0$.
This convention is very common in Banach lattice theory but might be somewhat uncommon for people who work mostly on function spaces.
A vector $f \in E_+$ is called a \emph{quasi-interior point of $E_+$} if the so-called \emph{principal ideal} $E_f \coloneqq \bigcup_{c > 0} \{g \in E: \abs{g} \le c f\}$ is dense in $E$.
An operator $T \in \calL(E)$ is called \emph{positive} if $TE_+ \subseteq E_+$ and we denote this by $T \ge 0$. The dual space $E'$ of a Banach latticec $E$ is again a Banach lattice; for every functional $\varphi \in E'$ we have $\varphi \ge 0$ if and only if $\langle \varphi, f\rangle \ge 0$ for all $f \in E_+$. A functional $\varphi \in E'$ is called \emph{strictly positive} if $\langle \varphi, f \rangle > 0$ for all $f > 0$.

A Banach lattice $E$ is called an AL-space if $\norm{f+g} = \norm{f} + \norm{g}$ for all $f,g \in E_+$.
Every $L^1$-space is an AL-space and conversely every AL-space is isometrically Banach lattice isomorphic to $L^1(\Omega,\mu)$ 
for some (not necessarily $\sigma$-finite) measure space $(\Omega,\mu)$. 
A linear operator $T \in \calL(E)$ on an AL-space $E$ is called a \emph{Markov operator} 
if $T \ge 0$ and $\norm{Tf} = \norm{f}$ for all $f \in E_+$. On every AL-space $E$ there is a uniquely determined functional $\mathds{1} \in E'$ with the 
property $\langle \mathds{1}, f \rangle = \norm{f}$ for all $f \in E_+$; this functional is called the \emph{norm functional} on $E$. Clearly, a positive operator $T \in \calL(E)$ 
is a Markov operator if and only if $T'\mathds{1} = \mathds{1}$.

A Banach lattice $E$ is called a \emph{KB-space} if every norm bounded increasing sequence in $E$ is norm convergent. This is equivalent to the seemingly stronger condition that every norm bounded increasing net in $E$ 
is norm convergent (which easily follows from the fact that every KB-space is a band in its bi-dual and has order continuous norm;
see Theorem~2.4.12 and the paragraph after Definition~2.4.11 in \cite{meyer1991}). We shall frequently use the fact that every AL-space is a KB-space \cite[Cor~2.4.13]{meyer1991}.

For the sake of completeness we recall from the very beginning that a \emph{semigroup} on a Banach space $E$ is always understood to be a family $\cT = (T_t)_{t \in J} \subseteq \calL(E)$ 
where either $J = \bbN \coloneqq \{1,2,\cdots\}$ or $J = (0,\infty)$ and where the semigroup law $T_{t+s} = T_t T_s$ is fulfilled for all $s,t \in J$.
It is worth emphasising that we do not require any continuity of the mapping $t \mapsto T_t$ in case that $J = (0,\infty)$.
A vector $f \in E$ is called a \emph{fixed vector} (or a \emph{fixed point}) of $\cT$ of $T_tf = f$ for all $t \in J$ and a functional $\varphi \in E'$ is called a \emph{fixed functional} of $\cT$ if $T_t'\varphi = \varphi$ for all $t \in J$.
If $E$ is a Banach lattice, then the semigroup $\cT$ is called \emph{positive} if $T_t \ge 0$ for all $t \in J$. 
If $\cT$ is a positive semigroup on a Banach lattice $E$, then a vector $f \in E_+$ is called a \emph{super fixed vector} (or a \emph{super fixed point}) of $\cT$ if $T_tf \ge f$ for all $t \in J$. 
If $E$ is an AL-space, then $\cT$ is called a \emph{Markov semigroup} if $T_t$ is a Markov operator for every $t \in J$.

Let $\cT = (T_t)_{t \in J}$ be a semigroup on a Banach space $E$. Then  $\cT$ is called \emph{bounded} 
if $\sup_{t \in J} \norm{T_t} < \infty$; it is called \emph{locally bounded at $0$} if for one (equivalently all) $c \in J$ we have $\sup_{t \in (0,c] \cap J} \norm{T_t} < \infty$.
Note that if $J = \bbN$, then $\cT$ is automatically locally bounded at $0$. 

We will very often be concerned with the question whether a semigroup $\cT = (T_t)_{t \in J}$ on a Banach space $E$ is \emph{convergent} with 
respect to the strong operator topology or, on a few occasions, with respect to the operator norm. Here we always mean convergence as $t \to \infty$.

On some occasions in the article we make freely use of ultra powers of Banach spaces. Here we only recall the basics in order to fix the notation: 
let $E$ be a real or complex Banach space and let $\cU$ be a free ultra filter on $\bbN$. Let $c_{0,\cU}(\bbN;E) \subseteq \ell^\infty(\bbN;E)$ be given by
\begin{align*}
	c_{0,\cU}(\bbN;E) \coloneqq \{(f_n) \in \ell^\infty(\bbN;E): \lim_{\cU} f_n = 0\}.
\end{align*}
Then the space $E^\cU \coloneqq \ell^\infty(\bbN;E) / c_{0,\cU}(\bbN;E)$ is called the \emph{ultra power} of $E$ with respect to $\cU$. 
For every $f = (f_n) \in \ell^\infty(\bbN;E)$ we denote by $f^\cU = (f_n)^\cU \coloneqq f + c_{0,\cU}(\bbN;E)$ the equivalence class of $f$ in $E^\cU$; 
note that $\norm{f^\cU} = \lim_{\cU} \norm{f_n}$. If $f \in E$, then we denote by $f^\cU$ the equivalence class of the constant sequence $(f)_{n \in \bbN}$ in $E^\cU$. 
The mapping $E \to E^\cU$, $f \mapsto f^\cU$ is an isometric embedding of the Banach space $E$ into the Banach space $E^\cU$ and is called the \emph{canonical embedding} of $E$ into $E^\cU$. 
If $E$ is a Banach lattice, then so is $E^\cU$ and the canonical embedding is a lattice homomorphism in this case. 

Every operator $T \in \calL(E)$ has a lifting $T^\cU \in \calL(E^\cU)$ which is given by $T^\cU f^\cU \coloneqq (Tf_n)^\cU$ 
for every $f = (f_n) \in \ell^\infty(\bbN;E)$ and the mapping $\calL(E) \to \calL(E^\cU)$, $T \mapsto T^\cU$ is an isometric Banach algebra homomorphism.

\section{Three Abstract Convergence Theorems} \label{section:three-abstract-convergence-theorems}

In this section we prove three theorems on semigroup convergence which hold in the general setting of Banach spaces and which are abstract 
in the sense that they do not provide explicit criteria for the convergence of a semigroup. 
Instead, they clarify the connection between different types of convergence. 
These theorems, in particular the first one, will be very useful later on.

Our first theorem relates the convergence of a semigroup $\cT = (T_t)_{t \in (0,\infty)}$ to the convergence of the powers $(T_t)^n = T_{tn}$ as $n \to \infty$.

\begin{theorem} \label{thm:discrete-to-continuous}
	Let $E$ be a real or complex Banach space and let $\cT = (T_t)_{t \in (0,\infty)}$ be a semigroup on $E$ which is locally bounded at $0$.
	\begin{enumerate}[(a)]
		\item The semigroup $(T_t)_{t \in (0,\infty)}$ converges strongly if and only if for every $t \in (0,\infty)$ the semigroup $(T_{tn})_{n\in \bbN}$ converges strongly.
		\item The semigroup $(T_t)_{t \in (0,\infty)}$ converges with respect to the operator norm if and only if for every $t > 0$ the semigroup $(T_{tn})_{n \in \bbN}$ converges with respect to the operator norm.
	\end{enumerate}
	If the equivalent assertions in (a) or (b) hold, then all occurring limits coincide.
\end{theorem}

The method that we use to prove this theorem is essentially known from \cite[the proof of Thm~4]{lotz1986}. Yet, we could not find the theorem in its full generality in the literature. We need the following simple lemma whose proof is taken from \cite[the proof of Theorem~VI.1.1]{doob1953}.

\begin{lemma} \label{lem:projection-semigroup}
	Let $E$ be a real or complex Banach space and let $\cP = (P_t)_{t \in (0,\infty)}$ be a semigroup on $E$ such that $P_t$ is a projection for every $t \in (0,\infty)$. 
	Then $\cP$ is constant, i.e.\ $P_t = P_s$ all $t,s \in (0,\infty)$.
\end{lemma}
\begin{proof}
	Let $\tau,t \in (0,\infty)$ and choose $n \in \bbN$ such that $\frac{\tau}{n} < t$. Then we have
	\begin{align*}
		P_t = P_{t - \frac{\tau}{n}} P_{\frac{\tau}{n}} = P_{t - \frac{\tau}{n}} P_{(n+1) \frac{\tau}{n}} = P_{t+\tau}.
	\end{align*}
	This proves the lemma.
\end{proof}

\begin{proof}[Proof of Theorem~\ref{thm:discrete-to-continuous}]
	We only show assertion (a) because the proof of (b) is virtually the same. Throughout, we may assume that the scalar field is complex, 
	since otherwise  we replace $E$ with a Banach space complexification of $E$. 
	The implication ``$\Rightarrow$'' is obvious. To prove the converse implication assume that for every $t \in (0,\infty)$ the semigroup $(T_{tn})_{n \in \bbN}$ is strongly convergent.
	
	We first show that the semigroup $\cT$ is in fact bounded. 
	Since $\cT$ is locally bounded at $0$, we have $C \coloneqq \sup_{t \in (0,1]} \norm{T_t} < \infty$. 
	Moreover it follows from the Uniform Boundedness Principle that $D \coloneqq \sup_{n \in \bbN} \norm{T_{n}} < \infty$ since $(T_n)$ converges strongly as $n \to \infty$. 
	Hence we have $\norm{T_t} \le CD$ for every $t \in [0,\infty)$.

	Now we show that the strong limits $P_t \coloneqq \lim_{n \to \infty} T_{tn}$ coincide for all $t > 0$. 
	Clearly, each operator $P_t$ is a projection. Moreover we have
	\begin{align*}
		P_sP_t = \lim_{n \to \infty} T_{sn} \lim_{n \to \infty} T_{tn} = \lim_{n \to \infty} T_{sn}T_{tn} = \lim_{n \to \infty} T_{(s+t)n} = P_{s+t},
	\end{align*}
	for all $s,t \in (0,\infty)$. Here we used that on bounded subsets of $\calL(E)$ the operator multiplication is jointly continuous with respect to the strong operator topology. 
	Hence, $(P_t)_{t \in (0,\infty)}$ is a semigroup. 
	We define $P \coloneqq P_t$ for some $t > 0$ and 
	we conclude from Lemma~\ref{lem:projection-semigroup} that $P = P_t$ for \emph{every} $t>0$. 
	Hence, we proved that $\lim_{n \to \infty} T_{tn} = P$ for all $t > 0$ and this implies that $T_tP = \lim_{n \to \infty} T_t T_{tn} = P$ for every $t > 0$. 

	Now, let $M \coloneqq \sup_{t \in (0,\infty)} \norm{T_t} < \infty$, $x \in E$ and $\varepsilon > 0$. 
	Observe that we have $\norm{T_{n_0}x - Px} \le \eps$ for some sufficiently large $n_0 \in \bbN$. 
	For all $t > n_0$ we thus obtain
	\begin{align*}
		\norm{T_tx - Px} = \norm{T_{t-n_0}T_{n_0}x - T_{t-n_0}Px} \le M \norm{T_{n_0}x - Px} \le M \eps.
	\end{align*}
	This proves that $T_tx \to Px$ as $t \to \infty$.
\end{proof}

If the scalar field is real and the limit operators under consideration are of rank one,
then Theorem~\ref{thm:discrete-to-continuous} can be slightly improved. In this case, one only needs to show that at least \emph{one} of the embedded discrete semigroups $(T_{nt})_{n \in \bbN}$ 
converges in order to obtain convergence of the entire semigroup. This is the content of our second abstract convergence theorem. 
It was inspired by \cite[Proposition~3.1]{malczak1992} where part (a) of this result was proved for the special case of a Markov semigroup on an AL-space.

\begin{theorem} \label{thm:discrete-to-continuous-rank-1}
	Let $E$ be a real Banach space and let $\cT = (T_t)_{t \in (0,\infty)}$ be a semigroup on $E$ which is locally bounded at $0$. Fix $t_0 \in (0,\infty)$.
	\begin{enumerate}[(a)]
		\item The semigroup $(T_t)_{t \in (0,\infty)}$ converges strongly to a rank-$1$ operator if and only if the semigroup $(T_{nt_0})_{n \in \bbN}$ converges strongly 
		to a rank-$1$ operator. In this case, both limit operators coincide.
		\item The semigroup $(T_t)_{t \in (0,\infty)}$ converges to a rank-$1$ operator with respect to the operator norm if and only if 
		the semigroup $(T_{nt_0})_{n \in \bbN}$ converges to a rank-$1$ operator with respect to the operator norm. In this case, both limit operators coincide.
	\end{enumerate}
\end{theorem}
\begin{proof}
	Again, we only prove assertion (a) because the proof of assertion (b) is virtually the same. 
	The implication ``$\Rightarrow$'' is obvious and we clearly have $\lim_{t \to \infty} T_t = \lim_{n \to \infty} T_{nt_0}$ if both limits exist.
	
	To prove the implication ``$\Leftarrow$'', assume that $(T_{nt_0})_{n \in \N}$ converges strongly to a rank-$1$ operator 
	$P$ as $n \to \infty$. Using that $\cT$ is locally bounded at $0$, one shows as in the proof of Theorem~\ref{thm:discrete-to-continuous} that the semigroup $\cT$ is bounded.

	Since $P$ is of rank $1$, we find a vector $f_0 \in E \setminus \{0\}$ and a functional $\varphi \in E' \setminus \{0\}$ such that 
	$P = \varphi \otimes f_0$. Clearly, $f_0$ is a fixed vector and $\varphi$ a fixed functional of $T_{t_0}$. 
	We  show next that $f_0$ is a fixed vector of each operator $T_t$. To this end
	define $\alpha_t \coloneqq \langle \varphi, T_tf_0 \rangle \in \bbR$ for every $t\in(0,\infty)$. Since
	\begin{align*}
		T_t f_0 = T_t T_{nt_0}f_0 = T_{nt_0}T_tf_0 \to \langle \varphi, T_t f_0 \rangle f_0 \quad \text{as } n \to \infty,
	\end{align*}
	we have $T_t f_0 = \alpha_t f_0$ for all $t\in(0,\infty)$ and we have to prove that $\alpha_t = 1$ for each time $t$. 
	If we had $\abs{\alpha_t} > 1$ for some $t\in(0,\infty)$, this would imply that $\norm{T_{nt}f_0} = \abs{\alpha_t}^n \norm{f_0} \to \infty$ as $n \to \infty$, 
	which contradicts the boundedness of $\cT$. Now assume that $\abs{\alpha_t} < 1$ for some $t\in(0,\infty)$. Then $T_{tn}f_0 \to 0$ as $n \to \infty$ and,
	since $\cT$ is bounded, we even obtain that $T_sf_0 \to 0$ as $s \to \infty$. This contradicts the fact that $f_0$ is a fixed point of $T_{t_0}$. 
	Hence, $\abs{\alpha_t} = 1$ for all $t\in(0,\infty)$.
	Since the Banach space $E$ is real, $\alpha_t \in \{-1,1\}$ for all $t\in(0,\infty)$. For each $t$ we have
	\begin{align*}
		\alpha_t f_0 = T_t f_0 = T_{\frac{t}{2}} T_{\frac{t}{2}} f_0 = \alpha_{\frac{t}{2}}^2 f_0
	\end{align*}
	and thus, we may conclude from $f_0 \not= 0$ that $\alpha_t = \alpha_{\frac{t}{2}}^2 > 0$. Hence, $\alpha_t = 1$ for all $t\in(0,\infty)$, which proves that $f_0$ is a fixed vector of $\cT$. 
	
	Now, let $f \in E$ be an arbitrary vector and let $\varepsilon > 0$. 
	Choose $n_0 \in \bbN$ such that $\norm{T_{n_0t_0}f - \langle \varphi, f \rangle f_0} \le \eps$. For every $t > n_0t_0$ we then obtain
	\begin{align*}
		\norm{T_tf - \langle \varphi, f\rangle f_0} = \norm{T_{t-n_0t_0}T_{n_0t_0}f - \langle \varphi, f\rangle T_{t-n_0t_0}f_0} \le \sup_{t \in (0,\infty)} \norm{T_t} \cdot \varepsilon. 
	\end{align*}
	This proves that $(T_t)$ converges strongly to $\varphi \otimes f_0$ as $t \to \infty$.
\end{proof}

\begin{remarks}
(a) For the non-trivial implications in Theorem~\ref{thm:discrete-to-continuous-rank-1} the assumption that the limit operator be of rank $1$ cannot be dropped. For example, 
let $\cT = (T_t)_{t \in (0,\infty)}$ a be a non-constant periodic semigroup on $\bbR^2$ with period $t_0 > 0$. Then $(T_{nt_0})_{n \in \bbN}$ converges, but $(T_t)_{t \in (0,\infty)}$ does not converge.
		
(b) The assumption that the Banach space $E$ be real is essential in Theorem~\ref{thm:discrete-to-continuous-rank-1}. Indeed, the semigroup $(e^{it})_{t \in (0,\infty)}$
on $\bbC$ does not converge, while the embedded discrete semigroup $(e^{i2\pi n})_{n \in \bbN}$ is constantly $1$ and thus convergent. On the other hand, 
if $E$ is a complexification of a real Banach space $E_\bbR$ and the semigroup $\cT$ leaves $E_\bbR$ invariant, then Theorem~\ref{thm:discrete-to-continuous-rank-1} clearly remains true.
		
(c) In the case of positive semigroups one can prove the following time-discrete analogue of Theorem~\ref{thm:discrete-to-continuous-rank-1}: let $E$ be a Banach lattice and 
let $T \in \calL(E)$ be a positive operator. If there is an integer $m_0 \in \bbN$ such that $(T^{nm_0})_{n \in \bbN}$ converges strongly (or with respect to the operator norm) to a rank-$1$ 
operator, then $(T^n)_{n \in \bbN}$ converges strongly (or with respect to the operator norm) to the same operator. The proof is essentially the same as 
for Theorem~\ref{thm:discrete-to-continuous-rank-1} where the positivity assumption is used to ensure that none of the numbers $\alpha_n \coloneqq \langle \varphi, T^nf_0 \rangle$ equals $-1$.
Without the additional positivity assumption the assertion is false as the one-dimensional operator $T = -1 \in \calL(\bbR)$ shows. 
\end{remarks}

The third of our three abstract convergence theorems describes how the convergence of a semigroup on a Banach space $E$ with respect to the operator norm 
is related to strong convergence of the semigroup on an ultra power $E^\cU$.

\begin{theorem} \label{thm:strong-ultra-convergence-implies-operator-norm-convergence}
	Let $E$ be a real or complex Banach space and let $\cT = (T_t)_{t \in J}$ be a semigroup on $E$ which is locally bounded at $0$ where either $J = \bbN$ or $J = (0,\infty)$. 
	Let $\cU$ be a free ultra filter on $\bbN$ and let $P \in \calL(E)$. Then the following assertions are equivalent:
	\begin{enumerate}[(i)]
		\item $(T_t)$ converges to $P$ with respect to the operator norm.
		\item $(T_t^\cU)$ converges to $P^\cU$ with respect to the operator norm.
		\item $(T_t^\cU)$ converges strongly to $P^\cU$.
	\end{enumerate}
\end{theorem}

In the context of semigroup convergence on ultra powers we also refer the reader to the paper \cite{nagel1993} where so-called \emph{superstable} operators are studied.
Before we proceed with the proof of Theorem~\ref{thm:strong-ultra-convergence-implies-operator-norm-convergence} we stress some important observations.

\begin{remarks} \label{rem:remarks-on-strong-convergence-on-ultra-powers}
	(a) Theorem~\ref{thm:strong-ultra-convergence-implies-operator-norm-convergence} does \emph{not} assert that strong convergence of the lifted 
	semigroup $\cT^\cU \coloneqq (T_t^\cU)_{t \in J}$ implies operator norm convergence of the original semigroup $\cT$. 
	In fact, the theorem only makes this assertion under the additional assumption that the strong limit $\lim_{t \to \infty} T_t^\cU \in \calL(E^\cU)$ is of the form $P^\cU$ for an operator $P \in \calL(E)$. 
	Without this assumption the assertion is false as the following counterexample shows.
	
	Let $E \coloneqq \ell^2 \coloneqq \ell^2(\bbN;\bbC)$ and let $A \in \calL(\ell^2)$ be the multiplication operator with symbol $(-\frac{1}{n})_{n \in \bbN}$.
	For each $t\geq 0$ define $T_t \coloneqq e^{tA}$.  Then it follows from~\cite[Thm~5.5.5(b)]{arendt2011} that the $C_0$-semigroup $\cT\coloneqq (T_t)_{t\in [0,\infty)}$ converges strongly to $0$. 
	However, since the spectral bound of $A$ equals $0$, the semigroup $\cT$ does not converge to $0$ with respect to the operator norm and 
	hence it does not converge with respect to the operator norm at all.
	
	On the other hand, let $\cU$ be any free ultra filter on $\bbN$ and consider the semigroup $\cT^\cU = (T_t^\cU)_{t\in [0,\infty)}$.
	Then $\cT^\cU$ converges with respect to the strong operator topology in $\calL(E^\cU)$. 
	To see this, first note that $\cT^\cU$ is a $C_0$-semigroup on $(\ell^2)^\cU=E^\cU$ with generator $A^\cU$. 
	Moreover, $E^\cU$ is a Hilbert space and the point spectrum of $A^\cU$ intersects the imaginary axis only in $0$. 
	Therefore, \cite[Thm~5.5.6(b)]{arendt2011} implies that $\cT^\cU$ is strongly convergent
	and Theorem~\ref{thm:strong-ultra-convergence-implies-operator-norm-convergence} shows that the limit operator cannot be of the form $P^\cU$ for $P \in \calL(E)$.
		
	(b) The semigroup structure of the net $(T_t)_{t \in J}$ is essential in Theorem~\ref{thm:strong-ultra-convergence-implies-operator-norm-convergence}, 
	i.e.\ the implication ``(iii) $\Rightarrow$ (i)'' does not hold for arbitrary sequences (or nets) in $\calL(E)$.
	In order to construct a counterexample, let $E = \ell^2 \coloneqq \ell^2(\bbN;\bbC)$ and, for every $n \in \bbN$, 
	define $T_n \in \calL(\ell^2)$ to be the multiplication operator whose symbol is the $n$-the canonical unit vector $e_n \in \ell^\infty$.
	Clearly, $(T_n)$ converges strongly to $0$ as $n \to \infty$, but $\norm{T_n} = 1$ for all indices $n$, so $(T_n)$ does not converge with respect to the operator norm. 
	Nevertheless, for any free ultra filter $\cU$ on $\bbN$ the sequence of lifted operators $(T_n^\cU)$ on the ultra power $E^\cU = (\ell^2)^\cU$ converge strongly to $0$. 
	To see this, let $f = (f_k)_{k \in \bbN} \subseteq \ell^2$ be a bounded sequence. 
	Aiming for a contradiction, we assume that $(T_n^\cU f^\cU)$ does not converge to $0$ in $E^\cU$.
	Then we can find an $\varepsilon > 0$ and a strictly increasing sequence of integers $(n_m)_{m \in \bbN}$ such that $\norm{T_{n_m}^\cU f^\cU} \ge \varepsilon$ for every index $m$. 
	Hence, we have for every $m \in \bbN$
	\begin{align*}
		\lim_{k,\cU} \abs{f_k(n_m)} = \lim_{k,\cU} \norm{T_{n_m}f_k} = \norm{T_{n_m}^\cU f^\cU} \ge \varepsilon.
	\end{align*}
	Now, let $M \subseteq \bbN$ be an arbitrary finite set. 
	For every $m \in M$ we can find a set $U_m \in \cU$ such that $\abs{f_k(n_m)} \ge \frac{\varepsilon}{2}$ for all $k \in U_m$. 
	Since $\cU$ is a filter, the intersection $\bigcap_{m \in M} U_m$ is non-empty, so we can find a number $k_M \in \bigcap_{m \in M} U_m$, 
	and this number fulfils $\abs{f_{k_M}(n_m)} \ge \frac{\varepsilon}{2}$ for all $m \in M$. 
	Since all $n_m$ are distinct, this implies
	\begin{align*}
		\norm{f_{k_M}}^2 \ge \abs{M} \left(\frac{\varepsilon}{2}\right)^2.
	\end{align*}
	This is a contradiction as the sequence $(f_k)$ is bounded in $\ell^2$. 
	Thus $(T_n^\cU)_{n \in \bbN}$ converges strongly to $0$.
\end{remarks}

\begin{proof}[Proof of Theorem~\ref{thm:strong-ultra-convergence-implies-operator-norm-convergence}]
	We may assume throughout the proof that the scalar field is complex; otherwise, we replace $E$ with a Banach space complexification of $E$.
	
	The implications ``(i) $\Rightarrow$ (ii)'' and ``(ii) $\Rightarrow$ (iii)'' are obvious. 
	To prove the non-trivial implication ``(iii) $\Rightarrow$ (i)''
	it suffices to consider the case $J=\N$; the assertion for $J = (0,\infty)$ then follows from Theorem~\ref{thm:discrete-to-continuous}(b).
	So we have $\cT = (T^n)_{n \in \bbN}$ for $T \coloneqq T_1 \in \calL(E)$. 
	
	Clearly, $(T^n)$ is strongly convergent to $P$. Thus $P$ is a projection and we have $TP = P$. Define $Q \coloneqq \id - P$. The projection $Q^\cU$ commutes with $T^\cU$ and we have
	\begin{align}
		\label{eq:convergence-to-zero-on-ultra-power}
		\big((TQ)^\cU\big)^n = (T^\cU Q^\cU)^n = (T^\cU)^nQ^\cU \to 0
	\end{align}
	with respect to the strong operator topology on $\calL(E^\cU)$ as $n \to \infty$. We show that this implies that $r(TQ) < 1$. 
	Since $TQ$ is power bounded, we have $r(TQ)\leq 1$. 
	Now assume that $r(TQ) = 1$ and pick a  spectral value $\lambda \in \sigma(TQ)$ with $\abs{\lambda} = r(TQ) = 1$. 
	As $\lambda \in \partial \sigma(TQ)$, the value $\lambda$ is an approximate eigenvalue of $TQ$ and hence an eigenvalue of $(TQ)^\cU$ \cite[Thm~V.1.4(ii)]{schaefer1974}.
	This contradicts~\eqref{eq:convergence-to-zero-on-ultra-power}.
	
	Therefore, $r(TQ)<1$ and this implies that $\norm{T^nQ} = \norm{(TQ)^n} \to 0$ as $n\to\infty$. Thus we have
	\begin{align*}
		T^n = T^n(P+Q) = P + T^nQ \to P \quad \text{as } t\to\infty
	\end{align*}
	with respect to the operator norm. 
\end{proof}

\section{The Lasota--Yorke Theorem Revisited} \label{section:lasota-yorke}

In this section we revisit a famous theorem of Lasota and Yorke, 
which asserts that a Markov semigroup on an $L^1$-space admitting a non-zero lower bound is strongly convergent.
It was originally proved in \cite[Thm~2, Rem~3]{lasota1982} for $L^1$-spaces over $\sigma$-finite measures spaces. 
However, the theorem is also true without any restriction on the measure space (see Theorem~\ref{thm:lasota-yorke} below or \cite[Thm~15]{emelyanov2004a}). 
This is not only a nice gadget, but it is crucial to several of our following results in whose proofs $L^1$-spaces occur in an abstract way 
(e.g.\ as ultra powers), so that we cannot ensure that the underlying measure space is $\sigma$-finite.

Recall that a Banach lattice $E$ is isometrically Banach lattice isomorphic to $L^1(\Omega,\mu)$ for some measure space 
$(\Omega,\mu)$ if and only if $E$ is a so-called \emph{AL-space}, meaning that the norm on $E$ is additive on the positive cone. 
We prefer not to use the representation of vectors in such spaces as integrable functions in our proofs but to argue with the abstract properties of AL-spaces, instead. 
This has certain theoretical advantages: on the one hand it becomes much clearer which properties of the space are used in the proofs; 
on the other hand we do not have to care about any measurability questions or about any problems that might be caused by non-$\sigma$-finite measures. 

Let us start by recalling the definition of a lower bound:

\begin{definition} \label{def:lower-bound-norm}
	Let $E$ be a Banach lattice and let $\cT = (T_t)_{t \in J}$ be a semigroup of positive operators on $E$, 
	where either $J = \bbN$ or $J = (0,\infty)$. 
	A vector $h \in E_+$ is called a \emph{lower bound} for $\cT$ if $\norm{(T_tf - h)^-} \to 0$ as $t \to \infty$ for every $f \in E_+$ with $\norm{f}=1$.
\end{definition}

While $0$ is always a lower bound, the Lasota-Yorke theorem asserts that the existence of a non-zero lower bound implies strong convergence of the semigroup. 
Here, we state the theorem in a version for AL-spaces, where the semigroup might be defined for discrete times or on the time interval $(0,\infty)$. 

\begin{theorem}[Lasota--Yorke] \label{thm:lasota-yorke}
	Let $E$ be an AL-space and let $\cT = (T_t)_{t \in J}$ be a Markov semigroup on $E$ where either $J = \bbN$ or $J = (0,\infty)$. 
	Then the following assertions are equivalent:
	\begin{enumerate}[(i)]
		\item There exists a normalised fixed point $f_0 \in E_+$ of $\cT$ such that $\cT$ converges strongly to the rank-1 projection $P=\mathds{1}\otimes f_0$.
		\item There exists a non-zero lower bound for $\cT$.
	\end{enumerate}
\end{theorem}

In the literature, the theorem is usually stated on $L^1$-spaces over $\sigma$-finite measure spaces; a version on AL-spaces can be found in \cite[Thm~15]{emelyanov2004a}. 
Here we want to include a version of the original proof from \cite[Thm~2]{lasota1982} on AL-spaces. Although only minor adjustments are necessary to make the proof work on AL-spaces, we chose to include the complete proof here, for the following two reasons: 
on the one hand, it can be written down in a very concise way in the setting of AL-spaces; 
on the other hand, we use the proof as a blueprint for the technically more involved proof of Theorem~\ref{thm:ind-lower-bounds-with-norm-bounds-for-markov-sg} in Section~\ref{section:individual-lower-bounds}.
So the reader might find it convenient to first read the proof of Theorem~\ref{thm:lasota-yorke} before proceeding with the proof of Theorem~\ref{thm:ind-lower-bounds-with-norm-bounds-for-markov-sg}.

It is rather common in the literature to prove the Lasota--Yorke theorem first for the time discrete case. 
The case $J = (0,\infty)$ is then reduced to the first one by using a simpler version of Theorem~\ref{thm:discrete-to-continuous}(a); 
see e.g.\ \cite[Thm~5.6.2, Thm~7.4.1]{lasota1994} or \cite[Thm~3.2.1]{emelyanov2007}. 
On the other hand, in~\cite[Thm~1.1]{lasota1983} an argument is given where $J$ is allowed to be a more general subsemigroup of the positive real numbers. We also treat the cases $J = \bbN$ and $J = (0,\infty)$ at the same time.

Let us first show the following two simple lemmas:

\begin{lemma} \label{lem:hmax-uniform}
	Assume that in the situation of Theorem~\ref{thm:lasota-yorke} the condition (ii) is fulfilled and let $H\subseteq E_+$ denote the set of all lower bounds for $\cT$. 
	Then $H$ has a maximum $h_{\max}$ and this maximum fulfils $0<\norm{h_{\max}}\leq 1$ and $T_t h_{\max} = h_{\max}$ for all $t\in J$.
\end{lemma}
\begin{proof}
	One easily checks that $H$ is closed and $\cT$-invariant and that $\norm{h}\le 1$ for all $h \in H$. 
	Moreover, it follows from $(a - x \lor y)^- \le (a-x)^- \lor (a-y)^- \le (a-x)^- + (a-y)^-$ for all $a,x,y \in E$ that $h_1 \lor h_2 \in H$ for all $h_1,h_2 \in H$.
	Hence, $(h)_{h \in H}$ is a norm bounded increasing net in $E$ and since every AL-space is a KB-space, $h_{\max }\coloneqq \lim_{h \in H}h$ exists in $E$. 
	Clearly, $h_{\max} = \sup H$, $h_{\max} \in H$ and $\norm{h_{\max}} \le 1$. Since $H\neq \{0\}$ by assumption, we also have that $h_{\max}>0$. 
	Moreover, for every time $t\in J$ we have $T_th_{\max} \in H$ and thus $T_th_{\max} \le h_{\max}$. 
	Yet, as $T_t$ is a Markov operator and as the norm on $E$ is strictly monotone, we conclude that in fact $T_th_{\max} = h_{\max}$.
\end{proof}

\begin{lemma} \label{lem:asymptotic-domination-of-normalised-vectors-implies-convergence}
	Let $E$ be an AL-space, let $h \in E_+$ of norm $1$ and let $(f_j)_{j \in J} \subseteq E_+$ be a net of vectors of norm $1$. 
	Assume that 
	\begin{align*}
		\lim_j \norm{(f_j - h)^-} = 0.
	\end{align*}
	Then $(f_j)$ converges to $h$.
\end{lemma}
\begin{proof}
	It suffices to show that $\lim_{j} \norm{(f_j-h)^+}= 0$. We have
	\begin{align*}
		1 & = \norm{f_j} = \norm{(f_j - h)^+ + h - (f_j - h)^-} \\
		& \ge \norm{(f_j - h)^+ + h} - \norm{(f_j - h)^-} = \norm{(f_j - h)^+} + \norm{h} - \norm{(f_j - h)^-},
	\end{align*}
	where the last equality follows from the fact that the norm is additive on $E_+$.
	Since $\norm{h} = 1$ it follows that $\norm{(f_j - h)^+} \le \norm{(f_j - h)^-} \to \infty$, which proves the assertion.
\end{proof}

\begin{proof}[Proof of Theorem~\ref{thm:lasota-yorke}]
	(i) $\Rightarrow$ (ii): If (i) one holds, then $\lim_{t \to \infty} T_t f = f_0$ for all $f \in E_+$ of norm $1$. 
	Hence, $f_0$ is a non-zero lower bound for $\cT$.

	(ii) $\Rightarrow$ (i): Assume that (ii) holds. Let $h_{\max}>0$ be the maximal lower bound of $\cT$ given by Lemma \ref{lem:hmax-uniform}
	and recall that $T_th_{\max}=h_{\max}$ for all $t\in J$.  We show  that  $\norm{h_{\max}}=1$.
	Indeed, suppose to the contrary that $\delta \coloneqq 1 - \norm{h_{\max}} > 0$. 
	We show that $(1+\delta)h_{\max}$ is then also a lower bound for $\cT$.
	So let $f \in E_+$, $\norm{f} = 1$ and let $\varepsilon > 0$. 
	There is a time $t_0\in J$ and a positive vector $e_{t_0}$ such that $T_{t_0} f + e_{t_0} \ge h_{\max}$ and $\norm{e_{t_0}} < \varepsilon$.
	The vector $g \coloneqq T_{t_0} f + e_{t_0} - h_{\max}$ is positive and has norm $\norm{g} \ge 1 - \norm{h_{\max}}= \delta$. 
	For all sufficiently large times $s\in J$, $s \ge s_0$ say, there is a vector $\tilde e_s$ such that $\norm{\tilde e_s} < \varepsilon$ and $T_sg + \tilde e_s \ge \norm{g} h_{\max} \ge \delta h_{\max}$.
	Hence we obtain for all $s \ge s_0$
	\begin{align*}
		T_sT_{t_0}f = T_sg + T_sh_{\max} - T_s e_{t_0} \ge (1+\delta)h_{\max} - T_s e_{t_0} - \tilde e_s.
	\end{align*}
	Since $\norm{T_s e_{t_0} + \tilde e_s} < 2 \varepsilon$, we have shown that $\norm{(T_tf-(1+\delta)h_{\max})^-} < 2\varepsilon$ for  all $t \ge t_0 + s_0$. 
	Since $f$ was arbitrary, this shows that $(1+\delta)h_{\max}$ is indeed a lower bound for $\cT$ which contradicts the maximality of $h_{\max}$. 
	We therefore proved that $\norm{h_{\max}} = 1$.
	
	Finally, let $f \in E_+$ with $\norm{f} = 1$. 
	Since $T_tf$ is of norm $1$ for all $t\in J$ and asymptotically dominates the vector $h_{\max}$ which also has norm $1$, 
	we conclude from Lemma~\ref{lem:asymptotic-domination-of-normalised-vectors-implies-convergence} that $T_tf \to h_{\max} = \langle \mathds{1}, f \rangle h_{\max}$ as $t \to \infty$. 
	By linearity $\lim T_tf = \langle \mathds{1}, f \rangle h_{\max}$ actually holds for all $f \in E$ which proves that (i) is true.
\end{proof}

It is natural to ask whether the Lasota-Yorke theorem remains true for semigroups which are not Markov but merely bounded 
(where, of course, we have to modify assertion (i) appropriately such that the trivial implication ``(i) $\Rightarrow$ (ii)'' remains true). 
The answer to this question is ``yes''.
For contractive semigroups in the time discrete case this was proved by Zalewska-Mitura in \cite[Thm~2.1]{zalewska-mitura1994}; 
by a somewhat technical procedure she reduced the assertion to the case of Markov operators. On the other hand, Komorn\'{i}k noted in \cite[Cor~5.1]{komornik1993} 
that the Lasota-Yorke theorem remains true even for power bounded positive operators; he derived this from a result about quasi-constrictive operators; see also \cite[Thm~3, Thm~4]{emelyanov2006} Using
Theorem~\ref{thm:discrete-to-continuous}(a), one can also treat the case $J = (0,\infty)$.

We are now going to give an alternative proof for the fact that Theorem~\ref{thm:lasota-yorke} remains true for bounded positive semigroups. Our argument is quite abstract and very short.

\begin{corollary} \label{cor:lasota-yorke-for-bounded-semigroups}
	Let $E$ be an AL-space and let $\cT = (T_t)_{t \in J}$ be a bounded semigroup on $E$ where either $J = \bbN$ or $J = (0,\infty)$. Then the following assertions are equivalent:
	\begin{enumerate}[(i)]
		\item $\cT$ converges strongly to the rank-1 projection $P=\phi \otimes f_0$, where 
		$f_0\in E_+$ is a normalised fixed vector of $\cT$ and $\phi\in E'$ is a positive fixed functional of $\cT$ such that $\applied{\phi}{f_0} =1$ and $\applied{\phi}{f} \geq \eps$ 
		for all $f \in E_+$ of norm $1$ and some constant $\varepsilon > 0$.
		\item There exists a non-zero lower bound $h$ for $\cT$.
	\end{enumerate}
\end{corollary}
\begin{proof}
	(i) $\Rightarrow$ (ii): If (i) holds, then $Pf \ge \varepsilon f_0$ for all $f \in E_+$ of norm $1$. Therefore, $\varepsilon f_0$ is a lower bound for $\cT$.
	
	(ii) $\Rightarrow$ (i): It suffices to consider the case $J=\N$; the assertion for $J = (0,\infty)$ then follows from Theorem~\ref{thm:discrete-to-continuous}(a).
	So assume that (ii) holds and that $\cT = (T^n)_{n \in \bbN}$ where $T \coloneqq T_1 \in \calL(E)$. 
	Fix an arbitrary Banach limit $L$ on $\ell^\infty(\bbN;\bbR)$ and define
		\[ \norm{f}_1 \coloneqq \langle L, (\norm{T^n\abs{f}})_{n \in \bbN} \rangle \]
	for all $f \in E$. Then $\norm{\argument}_1$ is an equivalent norm on $E$ such that $(E,\norm{\argument}_1)$ is again an AL-space. 
	Indeed, $\norm{\argument}_1$ is obviously a semi-norm on $E$ which is additive on the positive cone. 
	We clearly have $\norm{f}_1 \le \sup_{n \in \bbN} \norm{T^n} \norm{f}$ for all $f \in E$. 
	On the other hand, since $h$ is a lower bound of $\cT$, we 
	have $\norm{f}_1 \ge \norm{h} \norm{f}$ for every $f \in E$. 
	Thus, the semi-norm $\norm{\argument}_1$ is indeed equivalent to $\norm{\argument}$ and in particular, it is a norm.
	Moreover, $\norm{\argument}_1$ is even a lattice norm on $E$ and since $\norm{\argument}_1$ is additive on $E_+$, $(E,\norm{\argument}_1)$ is an AL-space as claimed. 
	
	Finally, since the Banach limit $L$ is shift invariant, it follows that $T$ is a Markov operator with respect to $\norm{\argument}_1$. 
	Hence, Theorem~\ref{thm:lasota-yorke} implies that $(T^n)$ converges strongly to the rank-$1$ projection $P\coloneqq \mathds{1}_1\otimes f_0$,
	where $f_0$ is a positive fixed point of $\cT$ such that $\norm{f_0}_1=1$  and $\mathds{1}_1$ denotes the norm functional on $(E,\norm{\argument}_1)$. 
	Note that $\norm{f_0} >0$ and that $\applied{\mathds{1}_1}{f} \geq \norm{h} >0$ for all $f\in E_+$ of norm $1$.
	Thus, replacing $f_0$ with $f_0 / \norm{f_0}$, assertion (i) follows with $\phi \coloneqq \norm{f_0} \cdot \mathds{1}_1$.
\end{proof}

As a further corollary we obtain a similar result for convergence with respect to the operator norm.

\begin{corollary} \label{cor:lasota-yorke-for-bounded-semigroups-and-convergence-in-operator-norm}
	Let $E$ be an AL-space and let $\cT = (T_t)_{t \in J}$ be a bounded semigroup on $E$ where either $J = \bbN$ or $J =(0,\infty)$. Then the following assertions are equivalent:
	\begin{enumerate}[(i)]
		\item The convergence in assertion (i) of Corollary \ref{cor:lasota-yorke-for-bounded-semigroups} holds with respect to the operator norm.
		\item There exists a vector $0 < h \in E$ such that 
		\[ \lim_{t\to\infty} \sup \bigl\{\norm{(T_tf-h)^-} : f \in E_+\text{, } \norm{f} = 1 \bigr\} =0.\]
	\end{enumerate}
\end{corollary}
\begin{proof}
	(i) $\Rightarrow$ (ii): If (i) holds, then we may conclude as in the proof of Corollary \ref{cor:lasota-yorke-for-bounded-semigroups} that 
	the semigroup has a unique normalised, positive fixed point $f_0$ and $Pf = f_0$ for all $f \in E_+$ of norm $1$. 
	The convergence in operator norm now implies assertion (ii) with $h=f_0$.
	
	(ii) $\Rightarrow$ (i): Fix a free ultra filter $\cU$ on $\bbN$ and note that the ultra power $E^\cU$ is again an AL-space.
	Using the fact that $\norm{(f_n)^\cU} = \lim_\cU \norm{f_n}$ for each $(f_n)^\cU \in E^\cU$, one can conclude from assertion (ii) 
	that the lifted element $h^\cU \in E^\cU$ is a non-zero lower bound of the lifted semigroup $\cT^\cU \coloneqq (T_t^\cU)_{t \in J}$. 
	Hence, by Corollary \ref{cor:lasota-yorke-for-bounded-semigroups}, 
	$(T_t^\cU)$ converges strongly to a positive rank-$1$ projection $Q \in \calL(E^\cU)$ as $t \to \infty$.
	In order to apply Theorem~\ref{thm:strong-ultra-convergence-implies-operator-norm-convergence} it remains to shows that $Q$ is of the form $P^\cU$ for some $P \in \calL(E)$. 
	Since $h$ is a non-zero lower bound for $\cT$, we know from Corollary \ref{cor:lasota-yorke-for-bounded-semigroups}
	that $\cT$ converges strongly to a positive rank-$1$ projection $P \in \calL(E)$ as $t \to \infty$; we show that $Q = P^\cU$.

	The operator $Q$ is of the form $Q = \psi \otimes g_0$, 
	where $g_0 \in E^\cU$ is the unique positive fixed point of $\cT^\cU$ of norm $1$ and where $\psi \in (E^\cU)'$ is the unique positive fixed functional of $\cT^\cU$ that fulfils $\langle \psi, g_0 \rangle = 1$.
	On the other hand, we obtain that the operator $P$ is of the form $P = \varphi \otimes f_0$, where $f_0 \in E$ is the unique positive fixed point of $\cT$ of norm $1$ 
	and where $\varphi \in E'$ is the unique positive fixed functional of $\cT$ that fulfils $\langle \varphi, f_0 \rangle = 1$. 
	Let $f_0^\cU \in E^\cU$ be the lifting of $f_0$ to $E^\cU$ and let $\varphi^\cU \in (E^\cU)'$ be given by 
	$\langle \varphi^\cU, f^\cU \rangle = \lim_\cU \langle \varphi, f_n \rangle$ for all $f^\cU = (f_n)^\cU \in E^\cU$.
	Then $f_0^\cU$ is a positive fixed vector of $\cT^\cU$ of norm $1$, so $f_0^\cU = g_0$. 
	Similarly, $\varphi^\cU$  is a positive fixed functional of $\cT^\cU$ which fulfils $\langle \varphi^\cU, g_0 \rangle = \langle \varphi^\cU, f_0^\cU \rangle = 1$, so $\varphi^\cU = \psi$.
	From $f_0^\cU = g_0$ and $\varphi^\cU = \psi$ one readily derives 
	that $P^\cU = (\varphi \otimes f_0)^\cU = \psi \otimes g_0 = Q$. Now the assertion follows from Theorem~\ref{thm:strong-ultra-convergence-implies-operator-norm-convergence}.
\end{proof}

Let us briefly outline an alternative way to prove Corollary~\ref{cor:lasota-yorke-for-bounded-semigroups-and-convergence-in-operator-norm}: 
using the concept of \emph{uniformly $\varepsilon$-overlapping operators} \cite[p.\,400]{bartoszek2008} one can conclude from \cite[Thm~1]{bartoszek1990} 
that the non-trivial implication of Corollary~\ref{cor:lasota-yorke-for-bounded-semigroups-and-convergence-in-operator-norm} holds at least for time-discrete Markov semigroups. 
Then one can use the Banach limit renorming trick from the proof of Corollary~\ref{cor:ind-lower-bounds-with-norm-bounds-for-bounded-sg} 
to conclude that the assertion is even true for time-discrete semigroups which are positive and bounded. Theorem~\ref{thm:discrete-to-continuous}(b) finally allows to obtain the case $J = (0,\infty)$, as well.

It is natural to ask whether the theorem of Lasota and Yorke or its generalisation in Corollary~\ref{cor:lasota-yorke-for-bounded-semigroups} still holds on more general Banach lattices. 
Our last result in this section shows that this is indeed true but for a trivial reason:
if a bounded positive semigroup admits a non-zero lower bound, then the underlying Banach lattice is automatically an AL-space.

\begin{theorem} \label{thm:lower-bounds-only-on-al-spaces}
	Let $E$ be a Banach lattice and let $\cT = (T_t)_{t \in J}$ be a bounded semigroup of positive operators on $E$, where $J = \bbN$ or $J = (0,\infty)$. 
	If there exists a lower bound $h \not= 0$ for $\cT$, then there exists an equivalent norm $\norm{\argument}_1$ on $E$ such that $(E,\norm{\argument}_1)$ is an AL-space.
\end{theorem}
\begin{proof}
	Choose a positive functional $\varphi \in E'_+$ such that $\langle \varphi, h \rangle > 0$. Moreover, let $\cU$ be an ultra filter on $J$ which contains the filter base $\{[t,\infty)\cap J : t \in J\}$ and define
	\begin{align*}
		\norm{f}_1 \coloneqq \lim_{\cU} \langle \varphi, T_t\abs{f} \rangle.
	\end{align*}
	Obviously, $\norm{\argument}_1$ is a semi-norm on $E$. Moreover, we have $\norm{f}_1 \le \norm{\varphi} \sup_{t \in J} \norm{T_t} \, \norm{f}$ and, 
	since $h$ is a non-zero lower bound for $\cT$, we also have $\norm{f}_1 \ge \norm{f} \langle \varphi, h \rangle$. Therefore,  $\norm{\argument}_1$ is 
	indeed a norm on $E$ and equivalent to $\norm{\argument}$.
	Now, it follows from the definition of $\norm{\argument}_1$ that $(E, \norm{\argument}_1)$ is a Banach lattice. 
	Since $\norm{\argument}_1$ is obviously additive on $E_+$, $(E,\norm{\argument}_1)$ is even an AL-space.
\end{proof}

\section{Individual Lower Bounds} \label{section:individual-lower-bounds}

In this section we adapt the Lasota-Yorke theorem to a situation where the lower bound of the semigroup is allowed to depend on the individual orbit.
The precise definition of such \emph{individual lower bounds} is a follows:

\begin{definition} \label{def:ind-lower-bounds}
	Let $E$ be a Banach lattice and let $\cT = (T_t)_{t \in J}$ be a positive semigroup on $E$ where either $J = \bbN$ or $J = (0,\infty)$. Let $f \in E_+$. 
	A vector $h_f \in E_+$ is called a \emph{lower bound} for $(\cT,f)$ if $\lim_{t \to \infty} \norm{(T_tf-h_f)^-} = 0$. 
\end{definition}

This notion was introduced by Ding in \cite[p.\ 312]{ding2003}. 
He proved that if an operator $T$ on an $L^1$-space over a $\sigma$-finite measure space $(\Omega,\mu)$ is a Frobenius-Perron operator associated with a non-singular transformation of $(\Omega,\mu)$ 
and if $((T^n)_{n \in \bbN},f)$ admits a non-zero lower bound for every $0 < f \in L^1(\Omega,\mu)$, then the semigroup $(T^n)_{n \in \bbN}$ is strongly convergent. 
To make this precise, we recall that a measurable mapping $\phi\colon \Omega\to \Omega$ is called \emph{non-singular}
if $\mu(\phi^{-1}(A))>0$ implies $\mu(A)>0$ for each measurable set $A\subseteq \Omega$. Given such a $\phi$,
there exists a uniquely determined operator $T\colon L^1(\Omega,\mu) \to L^1(\Omega,\mu)$ which satisfies
\[ \int_A Tf \dx\mu = \int_{\phi^{-1}(A)} f \dx \mu \]
for all $f\in L^1(\Omega,\mu)$ and every measurable $A\subseteq \Omega$, and this operator is called the \emph{Frobenius-Perron} operator associated with $\phi$. The precise statement of Ding's theorem now reads as follows. 

\begin{theorem}[Ding \mbox{\cite[Thm~1.1]{ding2003}}] \label{thm:ding}
	Let $(\Omega,\mu)$ be a $\sigma$-finite measure space and let $\phi\colon \Omega \to \Omega$ be measurable and non-singular. Let $T \in\calL(L^1(\Omega,\mu))$ 
	be the Frobenius--Perron operator associated with $\phi$ and set $\cT \coloneqq (T^n)_{n \in \bbN}$.
	
	If for every normalised $f \in L^1(\Omega,\mu)_+$ there exists a non-zero lower bound $h_f$ for $(\cT,f)$, then $(T^n)_{n\in\N}$ converges strongly.
\end{theorem}

It follows from Theorem~\ref{thm:discrete-to-continuous}(a) that an analogue of Ding's theorem also holds for semigroups $\cT = (T_t)_{t \in J}$ of Frobenius-Perron operators 
if $J = (0,\infty)$. We are going to give a generalisation of Theorem~\ref{thm:ding} in Theorem~\ref{thm:ding-abstract-and-for-bounded-sg} in the next section.
In the current section we devote ourselves to the following question: Ding asked in \cite[Rem~1.3]{ding2003} whether Theorem~\ref{thm:ding} remains true 
if we replace the Frobenius-Perron operator $T$ with an arbitrary Markov operator on $L^1(\Omega,\mu)$. 
The following examples shows that the answer is ``no'', in general. 

\begin{example} \label{ex:counterexample-to-dings-question}
	Let $\ell^1 \coloneqq \ell^1(\bbN_0)$ and denote by $e_k \in \ell^1$ for $k\in\N_0$ the $k$-th canonical unit vector.
	There is a positive operator $T \in \calL(\ell^1)$ with the following properties:
	\begin{enumerate}[(a)]
		\item $T$ is a Markov operator.
		\item The fixed space of $T$ is one-dimensional and spanned by $e_0$.
		\item For every $0 < f \in \ell^1$ there is a number $c_f > 0$ such that $T^nf \ge c_fe_0$ for all $n \in \bbN$. 
		In particular, $c_f e_0$ is a lower bound for $((T^n)_{n \in \bbN},f)$.
		\item No subsequence of $(T^n)$ is weakly convergent as $n \to \infty$.
	\end{enumerate}
	Indeed, define $h \coloneqq (1, \frac{1}{2}, \frac{1}{4},\cdots,\frac{1}{2^n},\cdots) \in \ell^\infty$ and let $M \in \calL(\ell^1)$ 
	be the multiplication operator with symbol $\mathds{1}_{\bbN_0} - h$. 
	Let $S \in \calL(\ell^1)$ denote the right shift and define the operator $T$ by $Tf \coloneqq \langle h, f\rangle e_0 + SMf$ for every $f \in \ell^1$. 
	Then a brief computation shows that $T$ is a Markov operator and $e_0$ is a fixed point of $T$.
	That the fixed space is one-dimensional follows from the fact that for any $f\in \ell^1$ and all $n\in\N$, $(T^n f)_k=0$ for all $1\leq k\leq n$.
	For $0 < f \in \ell^1$ define $c_f \coloneqq \langle h, f \rangle>0$. 
	Then $Tf \ge c_f e_0$ and, since $T$ is positive and $e_0$ is a fixed point of $T$, it follows that $T^nf \ge c_f e_0$ for all $n \in \bbN$.
	In order to verify property (d), one computes that 
		\[ T^ne_1 = (1-c_n)e_0 + c_n e_{n+1} \qquad \text{for all } n \in \bbN_0, \]
	where $c_n \coloneqq \prod_{k=1}^n (1-\frac{1}{2^k})$ for each $n \in \bbN_0$. 
	Since $(c_n)$ converges to a number $c > 0$ as $n \to \infty$ we conclude that no subsequence of $(T^ne_1)_{n \in \bbN_0}$ is norm convergent; as $\ell^1$ has Schur's property, it follows that no subsequence of $(T^ne_1)_{n \in \bbN_0}$ is even weakly convergent.
\end{example}

The same construction as in the above example appears in \cite[Exa~4.1]{komornik1993} where it is used as a counterexample for another question; see also \cite[Exa~3.1.28]{emelyanov2007}.

Example~\ref{ex:counterexample-to-dings-question} shows that the existence of non-zero \emph{individual} lower bounds for a semigroup is, in general, 
not sufficient for convergence. However, the situation is different if the individual lower bounds can be chosen such that their norm is bounded below. 
This is the content of the subsequent theorem, which is also the major step towards the proof of Theorem~\ref{thm:main-result-ind-lower-bounds}.

\begin{theorem} \label{thm:ind-lower-bounds-with-norm-bounds-for-markov-sg}
	Let $E$ be an AL-space and let $\cT = (T_t)_{t \in J}$ be a Markov semigroup on $E$ where either $J = \bbN$ or $J = (0,\infty)$. Then 
	the following assertions are equivalent:
	\begin{enumerate}[(i)]
		\item $\cT$ is strongly convergent.
		\item For every normalised $f \in E_+$ there exists a lower bound $h_f$ for $(\cT,f)$ such that $\inf_f \norm{h_f} > 0$.
	\end{enumerate}
\end{theorem}
\begin{proof}
	(i) $\Rightarrow$ (ii): If (i) holds and $f \in E_+$ is of norm $1$, then $h_f \coloneqq \lim_{t \to \infty} T_t f$ is also of norm $1$ and a lower bound for $(\cT,f)$.
	
	(ii) $\Rightarrow$ (i): Assume that (ii) holds and let  $\beta \coloneqq \inf_{f} \norm{h_f} > 0$.
	For each $f \in E_+$ of norm $1$ we denote by $H_f \subseteq E_+$ the set of all lower bounds for $(\cT,f)$. 
	By the same argument as in the proof of Lemma~\ref{lem:hmax-uniform} it follows that each $H_f$ has a maximum $h_{f,\max}$ which is a fixed point of $\cT$
	and one has that $\beta \le\norm{h_{f,\max}} \le 1$. 

	Now fix a normalised vector $f\in E_+$.
	It suffices to prove that $\norm{h_{f,\max}} = 1$, because then Lemma~\ref{lem:asymptotic-domination-of-normalised-vectors-implies-convergence} implies that $(T_tf)$ converges as $t \to \infty$.
	Suppose to the contrary that $\norm{h_{f,\max}}<1$ and let $\delta \coloneqq 1-\norm{h_{f,\max}} > 0$.

	Let $\eps>0$. We construct recursively an increasing sequence $(h_n)_{n \in \bbN_0}\subseteq E_+$ of fixed points of $\cT$ with the following properties:
	\begin{enumerate}[(a)]
		\item $\norm{h_n} \ge 1 - \delta \cdot (1-\beta)^n$ for all $n \in \bbN_0$.
		\item $\limsup_{t \to \infty} \norm{(T_tf - h_n)^-} < \varepsilon$ for every $n \in \bbN_0$.
	\end{enumerate}
	We set $h_0 \coloneqq h_{f,\max}$, which clearly fulfils conditions (a) and (b). 
	Now assume that $h_n$ has already been constructed for some $n\in\N_0$ and let $\tilde \varepsilon \coloneqq \limsup_{t \to \infty}\norm{(T_tf - h_n)^-}$.
	Since $\tilde \eps < \varepsilon$ by property (b), we have $\norm{(T_{t_0}f - h_n)^-} < (\tilde \varepsilon + \varepsilon)/2$ for some $t_0 \in J$. 
	Using the notation $g_n \coloneqq (T_{t_0}f - h_n)^+$ and $e_n \coloneqq (T_{t_0}f - h_n)^-$ we have
	\begin{align*}
		\norm{g_n} = \norm{T_{t_0}f - h_n + e_n} = \norm{T_{t_0}f + e_n} - \norm{h_n} \ge 1 - \norm{h_n},
	\end{align*}
	where the second equality holds since the norm is additive on $E_+$. We note in passing that it may happen that $1 - \norm{h_n} < 0$, in which case the above inequality is trivial.
	Let 
	\[ a_n \coloneqq 
		\begin{cases} 
			0 \quad & \text{if } g_n = 0 \\ 
			\norm{g_n} \cdot h_{\frac{g_n}{\norm{g_n}},\max} \quad & \text{if } g_n > 0;		
		\end{cases}
	\]
	then we have $\norm{(T_sg_n - a_n)^-} \to 0$ as $s \to \infty$. 
	Now define $h_{n+1} \coloneqq h_n + a_n$. Clearly, $h_{n+1} \ge h_n$ and $h_{n+1}$ is a fixed point of $\cT$ since $h_n$ and $a_n$ are fixed points. 
	Moreover, using again that the norm is additive on $E_+$, we obtain
	\begin{align*}
		\norm{h_{n+1}} & = \norm{h_n} + \norm{g_n} \norm{h_{\frac{g_n}{\norm{g_n}},\max}} \ge \norm{h_n} + (1-\norm{h_n})\beta = \beta + (1-\beta)\norm{h_n} \\
		& \ge \beta + (1-\beta)\big( 1 - \delta(1-\beta)^n \big) = 1 - \delta(1-\beta)^{n+1}.
	\end{align*}
	Hence, $h_{n+1}$ has property (a).  In order to verify property (b), let $s \in J$. Then
	\begin{align*}
		T_sT_{t_0}f - h_{n+1} &= T_s\big( g_n + h_n - e_n \big) - h_{n+1} \\
		&= \big(T_sg_n - a_n\big)^+ - \big(T_sg_n - a_n\big)^- - T_se_n,
	\end{align*}
	where we used that $h_n$ is a fixed point of $T_s$. Hence,
	\begin{align*}
		\norm{(T_sT_{t_0}f - h_{n+1})^-} & \le \norm{\big(T_sg_n - a_n\big)^-} + \norm{T_se_n} \\
		& \le \norm{\big(T_sg_n - a_n\big)^-} + \frac{\tilde \varepsilon + \varepsilon}{2} \to \frac{\tilde \varepsilon + \varepsilon}{2} < \eps
	\end{align*}
	as $s \to \infty$. This implies that $\norm{(T_tf - h_{n+1})^-} < \frac{1}{4}\tilde\eps + \frac{3}{4}\eps$ for all sufficiently large $t \in J$
	and hence $h_{n+1}$ has property (b). This completes the construction of the sequence $(h_n)$.

	For every $n\in\N_0$ it follows from property (b) and from
	\[ 0 \leq h_n = h_n - T_tf + T_tf \leq (T_tf-h_n)^- + T_tf\]
	that $\norm{h_n}\leq \eps+1$.
	Hence, the increasing sequence $(h_n)_{n \in \bbN_0}$ is bounded in norm and thus, since every AL-space is a KB-space, convergent to a vector $h_{f, \eps} \in E_+$. 
	It follows from (a) and (b) that $\norm{h_{f,\varepsilon}} \ge 1$ and $\limsup_{t \to \infty}\norm{(T_tf - h_{f,\varepsilon})^-} \le \varepsilon$.
	
	Choose a sequence $(\varepsilon_n)_{n \in \bbN} \subseteq (0,\infty)$ such that $\varepsilon \coloneqq \sum_{n=1}^\infty \varepsilon_n < \infty$. 
	For every $k \in \bbN$ and every $n \ge k$ define 
	\[ h_{f,k,n} \coloneqq h_{f,\varepsilon_k} \lor h_{f, \varepsilon_{k+1}} \lor \dots \lor h_{f, \varepsilon_n}.\]
	Using that $(a - x \lor y)^- \le (a-x)^- + (a-y)^-$ for all $a,x,y \in E$ one easily checks that 
	\[ \limsup_{t \to \infty} \norm{(T_tf - h_{f,k,n})^-} \le \sum_{j=k}^n \varepsilon_j \leq \sum_{j=k}^\infty \eps_j \]
	for all $k,n \in \N$, $k \le n$, and therefore $1 \le \norm{h_{f,k,n}} \le 1 + \sum_{j=k}^\infty \varepsilon_j$.
	For each fixed $k \in \bbN$, the increasing and norm bounded sequence $(h_{f,k,n})_{n \ge k}$ 
	converges to a vector $h_{f,k,\infty}$, which fulfils $1 \le \norm{h_{f,k,\infty}} \le 1 + \sum_{j=k}^\infty \varepsilon_j$
	and $\limsup_{t \to \infty} \norm{(T_tf - h_{f,k,\infty})^-} \le \sum_{j=k}^\infty \varepsilon_j$.
	The sequence $(h_{f,k,\infty})_{k \in \bbN} \subseteq E_+$ is decreasing and 
	thus convergent to a vector $h_{f,\infty,\infty}$ such that $\norm{h_{f,\infty, \infty}} = 1$ 
	and $\limsup_{t \to \infty} \norm{(T_tf - h_{f,\infty,\infty})^-} = 0$. The latter equality shows that $h_{f, \infty, \infty} \in H_f$ which
	contradicts the assumption that $\norm{h_{f,\max}}<1$.
	We therefore conclude that $\norm{h_{f,\max}}=1$ and the assertion follows from Lemma~\ref{lem:asymptotic-domination-of-normalised-vectors-implies-convergence}.
\end{proof}

For positive semigroups which are not Markov but merely bounded we can use the same argument as in Corollary~\ref{cor:lasota-yorke-for-bounded-semigroups} 
to obtain the following result which is a slightly enhanced version of Theorem~\ref{thm:main-result-ind-lower-bounds}.

\begin{corollary} \label{cor:ind-lower-bounds-with-norm-bounds-for-bounded-sg}
	Let $E$ be an AL-space and let $\cT = (T_t)_{t \in J}$ be a bounded positive semigroup on $E$ where either $J = \bbN$ or $J = (0,\infty)$. 
	Then for every $\varepsilon > 0$ the following assertions are equivalent:
	\begin{enumerate}[(i)]
		\item The semigroup $\cT$ converges strongly to an operator $P$ such that  $\norm{Pf} \ge \varepsilon \norm{f}$ for every $f \in E_+$.
		\item For every normalised $f \in E_+$ there is a lower bound $h_f$ for $(\cT,f)$ such that $\norm{h_f} \ge \varepsilon$.
	\end{enumerate}
\end{corollary}
\begin{proof}
	(i) $\Rightarrow$ (ii): If (i) holds and $f\in E_+$ is of norm $1$, then $Pf$ is a lower bound for $(\cT,f)$ which fulfils $\norm{Pf} \ge \varepsilon$.
	
	(ii) $\Rightarrow$ (i): It suffices to show that $\cT$ converges in the case $J=\N$, since the convergence for $J = (0,\infty)$ then follows from Theorem~\ref{thm:discrete-to-continuous}(a).
	So assume that (ii) holds and that $\cT = (T^n)_{n \in \bbN}$ where $T \coloneqq T_1 \in \calL(E)$. 
	Fix an arbitrary Banach limit $L$ on $\ell^\infty(\bbN;\bbR)$ and define
		\[ \norm{f}_1 \coloneqq \langle L, (\norm{T^n\abs{f}})_{n \in \bbN} \rangle \]
	for all $f \in E$.  As in the proof of Corollary~\ref{cor:lasota-yorke-for-bounded-semigroups} one shows that $\norm{\argument}_1$ is an equivalent norm on $E$,
	that $(E,\norm{\argument}_1)$ is also an AL-space and that $T$ is a Markov operator with respect to this norm. 
	Hence, it follows from Theorem~\ref{thm:ind-lower-bounds-with-norm-bounds-for-markov-sg} that $(T^n)$ converges strongly as $n \to \infty$.
	
	Now let $P \in \calL(E)$ be the strong limit of $\cT$, no matter whether $J = \bbN$ or $J = (0,\infty)$. 
	Then we have $\norm{Pf} \ge \norm{h_f} \ge \varepsilon$ for every normalised $f \in E_+$. This proves (i).
\end{proof}

It seems to be unclear whether one can obtain a version of Corollary~\ref{cor:ind-lower-bounds-with-norm-bounds-for-bounded-sg} for convergence in operator norm,
as we have for the Lasota-Yorke theorem in Corollary~\ref{cor:lasota-yorke-for-bounded-semigroups-and-convergence-in-operator-norm}.
Of course, we could try to employ an ultra power argument, again. However, since the limit operator of the lifted semigroup does not need to have rank~$1$, 
it is not clear whether it is given as a lifting of an operator on $E$; thus one cannot apply Theorem~\ref{thm:strong-ultra-convergence-implies-operator-norm-convergence}.

Let us derive one further corollary from the above results which deals with the convergence of asymptotically dominating semigroups.
Note that there are several results in the literature of the following type:
let $\cS$ and $\cT$ be positive semigroups on a Banach lattice such that $\cS$ dominates $\cT$ in some sense. 
If the dominating semigroup $\cS$ is strongly convergent, then it follows under appropriate technical assumptions that the dominated semigroup $\cT$ is also convergent; 
see e.g.\ \cite{emelyanov2001} and the references therein. In the following corollary we prove a contrary result: 
if the dominated semigroup $\cT$ converges and the limit operator behaves appropriately, then the dominating semigroup $\cS$ converges, too.

\begin{corollary} \label{cor:domination}
	Let $E$ be an AL-space and let $\cS = (S_t)_{t \in J}$ and $\cT = (T_t)_{t \in J}$ be two bounded positive semigroups on $E$
	where either $J = \bbN$ or $J = (0,\infty)$. Assume that $\cS$ asymptotically dominates $\cT$ in the sense that $\norm{(S_tf-T_tf)^-} \to 0$ as $t \to \infty$ for all $f \in E_+$.
	
	If $\cT$ is strongly convergent and its limit operator $P$ fulfils $\norm{Pf} \ge \varepsilon \norm{f}$ for all $f \in E_+$ and some $\varepsilon > 0$, 
	then $\cS$ is strongly convergent, too.
\end{corollary}
\begin{proof}
	Since $\cT$ fulfils assertion (i) of Corollary~\ref{cor:ind-lower-bounds-with-norm-bounds-for-bounded-sg},
	it also fulfils the equivalent assertion (ii). This obviously implies that $\cS$ fulfils assertion (ii) of the corollary, too, 
	and hence it fulfils the equivalent assertion (i).
\end{proof}

\section{Ding's Theorem Revisited} \label{section:individual-lower-bounds-without-a-norm-bound}

In the previous section we derived strong convergence of a semigroup under the assumption that it admits, in some sense, individual lower bounds. 
According to Example~\ref{ex:counterexample-to-dings-question} this is only possible under an additional assumption. 
In Theorem~\ref{thm:ind-lower-bounds-with-norm-bounds-for-markov-sg} we assumed in addition that the individual lower bounds are bounded below in norm. 
Ding, on the other hand, required in Theorem \ref{thm:ding} that the semigroup consists of Frobenius-Perron operators. 
In the next theorem, we replace this condition with the slightly more general requirement that the adjoint of each operator is a
lattice homomorphism and we only assume the semigroup to be bounded instead of Markov.

\begin{theorem} \label{thm:ding-abstract-and-for-bounded-sg}
	Let $E$ be an AL-space and let $\cT = (T_t)_{t \in J}$ be a bounded positive semigroup on $E$ where either $J = \bbN$ or $J = (0,\infty)$. 
	Suppose that for every $t \in J$ the adjoint operator $T_t' \in \calL(E')$ is a lattice homomorphism. Then the following assertions are equivalent:
	\begin{enumerate}[(i)]
		\item $\cT$ is strongly convergent and the limit operator $P \in \calL(E)$ fulfils $Pf > 0$ for every $f > 0$.
		\item For every $0<f \in E$ there exists a non-zero lower bound $h_f$ for $(\cT,f)$.
	\end{enumerate}
\end{theorem}

Let us briefly explain why Theorem~\ref{thm:ding-abstract-and-for-bounded-sg} generalises Theorem~\ref{thm:ding}.
Let $(\Omega,\mu)$ be a $\sigma$-finite measure space and let $\phi\colon \Omega \to \Omega$ be a measurable and non-singular mapping
with corresponding Frobenius-Perron operator $T \in \calL(L^1(\Omega,\mu))$. 
Since the measure space is $\sigma$-finite, the dual space of $L^1(\Omega,\mu)$ is given by $L^\infty(\Omega,\mu)$ and the adjoint $T'$ 
is the Koopman operator associated with $\phi$. 
Hence, every operator $(T^n)'=(T')^n$ is a lattice homomorphism and thus Theorem \ref{thm:ding} is a special case of Theorem~\ref{thm:ding-abstract-and-for-bounded-sg}.

On the other hand one can show, under appropriate assumptions on the measure space, that a Markov operator $T$ on $L^1(\Omega,\mu)$ whose adjoint is a lattice homomorphism is automatically a Frobenius-Perron operator.
Hence, the major novelty about Theorem~\ref{thm:ding-abstract-and-for-bounded-sg} is the fact that the semigroup $\cT$ is allowed to be merely bounded instead of Markov.

For the proof of Theorem~\ref{thm:ding-abstract-and-for-bounded-sg} we employ the main idea from the proof of \cite[Thm~1.1]{ding2003}, 
but we replace the usage of the explicit form of the operators with a general observation that we recall in Proposition \ref{prop:interval-preserving-if-the-adjoint-is-a-lattice-homomorphism}.
In a Banach lattice $E$ we denote the \emph{order interval} between any $f,g \in E$ by $[f,g] \coloneqq \{x \in E: f \le x \le g\}$.

\begin{proposition} \label{prop:interval-preserving-if-the-adjoint-is-a-lattice-homomorphism}
	Let $E$ be a Banach lattice with order continuous norm and let $T\in \calL(E)$ be such that the adjoint $T' \in \calL(E')$ is a lattice homomorphism. 
	Then we have $T[f,g] = [Tf,Tg]$ for all $f,g \in E$ with $f \le g$.
\end{proposition}
\begin{proof}
	This result can be found in \cite[Exercise~1.4.E2]{meyer1991}.
\end{proof}

Since we do not require the semigroup in Theorem \ref{thm:ding-abstract-and-for-bounded-sg} to be Markov, we also need the following adjusted version of Lemma \ref{lem:hmax-uniform}.

\begin{lemma} \label{lem:maximal-fixed-lower-bound}
	Let $\cT=(T_t)_{t\in J}$ be a bounded positive semigroup on an AL-space $E$ where either $J=\N$ or $J=(0,\infty)$ and suppose that for every  $0 < g\in E$ there exists
	a non-zero lower bound for $(\cT,g)$. 
	
	Let $0 < f\in E$ and denote by $H_{\mathrm{fix}}\subseteq E_+$ the set of all lower bounds for $(\cT,f)$ that are fixed points of $\cT$.
	Then $H_{\mathrm{fix}}$ has a non-zero maximum.
\end{lemma}
\begin{proof}
	Let $H\subseteq E_+$ denote the set of all lower bounds for $(\cT,f)$.
	As in the proof of Lemma~\ref{lem:hmax-uniform} on can show that $H$ is closed and $\cT$-invariant and that $h_1\vee h_2 \in H$ for all $h_1,h_2 \in H$.
	Moreover, since $\norm{h} \leq \sup_{t \in J} \norm{T_t}\norm{f}$ for all $h\in H$, it follows that there exists $h_{\max} = \sup H \in H$. For every $t \in J$ we have $T_t h_{\max} \in H$ and hence $T_t h_{\max} \le h_{\max}$. The net 
	$(T_th_{\max})_{t \in J}$ is thus decreasing and it converges to an element $h_0 \in E_+$ which is trivially a lower bound for $(\cT,f)$. 
	On the other hand, there exists a non-zero lower bound for $(\cT,h_{\max})$, which implies that $h_0 \not= 0$. 
	Hence, $h_0$ is a non-zero lower bound for $(\cT,f)$ and clearly a fixed point of $\cT$, i.e.\ $0 \not= h_0 \in H_{\mathrm{fix}}$.
	
	Let $\tilde H \coloneqq \{ h \in H : T_t h \geq h \text{ for all } t\in J\}$ denote the set of all lower bounds for $(\cT,f)$ which are  super fixed points  of $\cT$.
	Then $h_0 \in \tilde H$, $\tilde H$ is closed and $\cT$-invariant and again $h_1\vee h_2 \in \tilde H$ for all $h_1,h_2 \in \tilde H$. 
	Thus, the increasing and norm bounded net $(\tilde h)_{\tilde h \in \tilde H}$ converges to its supremum $0\neq \tilde h_{\max} \in \tilde H$.
	Since $\tilde H$ is invariant with respect to $\cT$, we have $T_t \tilde h_{\max} \le \tilde h_{\max}$ for all $t \in J$. 
	On the other hand, $\tilde h_{\max}$ is a super fixed point of $\cT$ and therefore $T_t \tilde h_{\max} \geq \tilde h_{\max}$.
	
	Thus, $\tilde h_{\max} \in H_{\mathrm{fix}}$ and since $H_{\mathrm{fix}} \subseteq \tilde H$, $\tilde h_{\max}$ is also the maximum of $H_{\mathrm{fix}}$.
\end{proof}

We can now prove Theorem~\ref{thm:ding-abstract-and-for-bounded-sg}.

\begin{proof}[Proof of Theorem~\ref{thm:ding-abstract-and-for-bounded-sg}]
	(i) $\Rightarrow$ (ii): If (i) holds and $0 < f \in E$, then $Pf$ is a non-zero lower bound for $(\cT,f)$.
	
	(ii) $\Rightarrow$ (i): It suffices to consider the case $J=\N$; the assertion for $J = (0,\infty)$ then follows from Theorem~\ref{thm:discrete-to-continuous}(a).
	So assume that (ii) holds and that $\cT = (T^n)_{n \in \bbN}$ for $T \coloneqq T_1 \in \calL(E)$. 
	Let $0 < f \in E$.  According to Lemma~\ref{lem:maximal-fixed-lower-bound} there exists a 
	maximum $h\neq 0$ of the set of all lower bounds for $(\cT,f)$ that are also fixed points of $\cT$.
	
	We are going to show that $T^nf \to h$ as $n \to \infty$. To this end, we first construct an increasing sequence $(f_n)_{n \in \bbN_0}$ with the following properties:
	\begin{enumerate}[(a)]
		\item $0 \le f_n \le f$ for each $n \in \bbN_0$.
		\item $T^nf_n = h \land T^nf$ for each $n \in \bbN_0$.
	\end{enumerate}
	Define $f_0 \coloneqq f \land h$, which clearly has the properties (a) and (b). 
	Now, assume that $f_n$ has already been defined for some $n\in\N_0$ and fulfils (a) and (b). 
	Then 
	\[ h\wedge T^{n+1}f = Th\wedge T(T^n f) \geq T(h\wedge T^n f) = T^{n+1}f_n \]
	and therefore $h\wedge T^{n+1}f \in [T^{n+1}f_n, T^{n+1}f] = T^{n+1}[f_n,f]$, where the equality of the two sets follows from Proposition~\ref{prop:interval-preserving-if-the-adjoint-is-a-lattice-homomorphism}.
	Thus, we find a vector $f_{n+1} \in [f_n,f]$ that satisfies conditions (a) and (b). This completes the construction of the sequence $(f_n)$.
	
	Since every AL-space is a KB-space, the increasing and norm bounded sequence $(f_n)$ converges to a vector $\hat f \in [0,f]$
	and for any $n\in\N$ we have
	\[ 0\leq (T^n \hat f - h)^- \leq (T^n f_n - h)^- = (T^nf\wedge h-h)^- = (T^nf-h)^-.\]
	Since $h$ is a lower bound for $(\cT,f)$, this implies that $h$ is also a lower bound for $(\cT,\hat f)$. 
	
	We show now that $\hat f=f$, so 
	suppose to the contrary that $\hat f < f$.
	Then, by Lemma~\ref{lem:maximal-fixed-lower-bound}, there exists a lower bound $h_1 > 0$ for $(\cT,f - \hat f)$ which is a fixed point of $\cT$.
	Hence, $h + h_1$ is a fixed point of $\cT$ and also a lower bound for $(\cT, \hat f+(f-\hat f)) = (\cT, f)$. This contradicts the maximality of $h$ and it thus follows that  $\hat f = f$.
	
	Now let $\varepsilon > 0$ and let $M \coloneqq \sup_{n \in \bbN} \norm{T^n}$. 
	For all sufficiently large $n$, $n \ge n_0$ say, we have $\norm{f - f_n} \le \varepsilon$ and $\norm{(T^nf-h)^-} \le \varepsilon$. Hence we obtain for all such $n$ that
	\begin{align*}
		\norm{T^nf - h} &\le \norm{T^nf_n - h} + M \varepsilon = \norm{h \land T^n f - h} + M \varepsilon \\ 
		&= \norm{(T^nf - h)^-} + M\varepsilon \le (M+1)\varepsilon.
	\end{align*}
	This proves that $T^n f \to h > 0$ as $n \to \infty$. Thus, assertion (i) holds.
\end{proof}

Let us briefly compare Theorem~\ref{thm:ding-abstract-and-for-bounded-sg} with Corollary~\ref{cor:ind-lower-bounds-with-norm-bounds-for-bounded-sg}: 
condition (ii) of the corollary contains the additional assumption that $\norm{h_f} \ge \varepsilon$ for every normalised $f \ge 0$; 
on the other hand, the limit operator in assertion (i) thus has to fulfil $\norm{Pf} \ge \varepsilon \norm{f}$ for every $f \ge 0$. 
In Theorem~\ref{thm:ding-abstract-and-for-bounded-sg} none of these additional properties occurs
and one might ask whether they are automatically fulfilled if the equivalent assertions (i) and (ii) hold. The following examples shows that the answer is ``no''.

\begin{example}
	Let $\ell^1 \coloneqq \ell^1(\N)$ and denote by $e_k \in \ell^1$ for $k\in \N$ the $k$-th canonical unit vector.
	Define $h \coloneqq (1,\frac{1}{2},\frac{1}{4}, \dots, \frac{1}{2^n},\dots) \in \ell^\infty$ and let $T\coloneqq h\otimes e_1 \in \cL(\ell^1)$.
	This operator has the following properties:
	\begin{enumerate}[(a)]
	\item The adjoint $T' = e_1\otimes h \in \cL(\ell^\infty)$ is a lattice homomorphism and $T$ is a projection; thus the semigroup $\cT = (T^n)_{n\in\N}=(T)_{n\in\N}$ 
	fulfils the assumptions of Theorem~\ref{thm:ding-abstract-and-for-bounded-sg}.
	\item Since $T$ is a projection, $\cT$ clearly converges strongly to the operator $P\coloneqq T$ and thus condition~(i) of Theorem~\ref{thm:ding-abstract-and-for-bounded-sg} is fulfilled.
	\item Since $\norm{Pe_k} = 2^{-k}$ for each $k\in\N$, there is no $\eps>0$ such that $\norm{Pf} \geq \eps \norm{f}$ for each $0\le f\in \ell^1$.
	\end{enumerate}
\end{example}

It is a natural question what happens if the assumption in Theorem~\ref{thm:ding-abstract-and-for-bounded-sg} that the adjoints $T_t'$ be lattice homomorphisms is replaced with the assumption that the operators $T_t$ themselves be lattice homomorphisms. The following theorem shows that such a semigroup can almost never have individual lower bounds.

\begin{theorem} \label{thm:ind-lower-bounds-for-sg-of-lattice-homomorphisms}
	Let $E$ be an AL-space and let $\cT = (T_t)_{t \in J}$ be a bounded positive semigroup on $E$ where either $J = \bbN$ or $J = (0,\infty)$. 
	Assume that every operator $T_t$ is a lattice homomorphism. Then the following assertions are equivalent:
	\begin{enumerate}[(i)]
		\item $\cT$ converges strongly to an operator $P$ that fulfils $Pf>0$ for every $f > 0$.
		\item We have $T_t = \id_E$ for all $t \in J$.
		\item For every $0<f \in E$ there exists a non-zero lower bound $h_f$ for $(\cT,f)$.
	\end{enumerate}
\end{theorem}
\begin{proof}
	(i) $\Rightarrow$ (ii): This follows from the more general Proposition~\ref{prop:semigroup-of-lattice-homomorphisms-identity-operator} below.

	(ii) $\Rightarrow$ (iii): This implication is obvious.

	(iii) $\Rightarrow$ (i): Assume that (iii) holds and fix a normalised vector $0<f \in E$.
	According to Lemma~\ref{lem:maximal-fixed-lower-bound} there exists a maximum $h\neq 0$ of the set of all lower bounds for $(\cT,f)$ that are also fixed points of $\cT$.
	We are going to show that $\norm{h} \ge 1$. Suppose to the contrary that $\norm{h} < 1$. Then we have $\norm{f \land h} \le \norm{h} < 1 \le \norm{f}$ and thus, 
	$f - f \land h > 0$. On the other hand, using that every operator $T_t$ is a lattice homomorphism and that $h$ is fixed point of $\cT$, 
	it follows from
	\[ (T_t(f\wedge h) - h)^- = (T_tf \wedge h-h)^- = (T_t f - h)^- \]
	that $h$ is a lower bound for $f \land h$. Again according to Lemma~\ref{lem:maximal-fixed-lower-bound} we find a fixed point $h_1 > 0$ of $\cT$ 
	which is a lower bound for $(\cT, f - f\land h)$. Thus, $h + h_1$ is a fixed point of $\cT$ and a lower bound for $(\cT,f \land h + (f - f \land h)) =(\cT, f)$,
	which contradicts the maximality of $h$.
	
	Hence, we have shown that for every normalised $f \in E_+$ there exists a lower bound $h$ for $(\cT,f)$ such that $\norm{h}\ge 1$. 
	Corollary~\ref{cor:ind-lower-bounds-with-norm-bounds-for-bounded-sg} now implies assertion (i).
\end{proof}

The equivalence of the assertions (i) and (ii) in the above theorem is a special case of the following proposition which holds on general Banach lattices.

\begin{proposition} \label{prop:semigroup-of-lattice-homomorphisms-identity-operator}
	Let $E$ be a Banach lattice and let $\cT = (T_t)_{t \in J}$ be a semigroup on $E$ where either $J = \bbN$ or $J = (0,\infty)$. 
	If each operator $T_t$ is a lattice homomorphism, then the following assertions are equivalent:
	\begin{enumerate}[(i)]
		\item $\cT$ converges strongly to an operator $P \in \calL(E)$ which fulfils $Pf > 0$ for all $f > 0$.
		\item We have $T_t = \id_E$ for all $t \in J$.
	\end{enumerate}
\end{proposition}
\begin{proof}
	The implication ``(ii) $\Rightarrow$ (i)'' is obvious.
	
	(i) $\Rightarrow$ (ii): Let $f \in E$. If (i) holds, then $P$ is also a lattice homomorphism and it thus follows from $f - Pf \in \ker P$ that even $\abs{f-Pf}\in \ker P$.
	Now condition~(i) implies that $\abs{f-Pf}=0$ and therefore $Pf=f$.  Hence, $f$ is a fixed point of $\cT$ and since $f$ was arbitrary we conclude that  $T_t = \id_E$ for every $t \in J$.
\end{proof}

Proposition~\ref{prop:semigroup-of-lattice-homomorphisms-identity-operator} has also has interesting consequences 
for semigroups of Frobenius-Perron operators that we state in the following two corollaries.

\begin{corollary} \label{cor:norm-convergence-of-sg-with-lattice-homomorphisms-as-adjoints}
	Let $E$ be a Banach lattice and let $\cT = (T_t)_{t \in J}$ be a semigroup on $E$ where either $J = \bbN$ or $J = (0,\infty)$. Assume 
	that for every time $t \in J$ the adjoint operator $T_t'$ is a lattice homomorphism and that $\cT$ has a fixed point $f_0 \ge 0$ which is a quasi-interior point of $E_+$.
	
	If $\cT$ converges with respect to the operator norm, then $T_t = \id_E$ for all $t \in J$.
\end{corollary}
\begin{proof}
	Assume that $\cT$ converges with respect to the operator norm to an operator $P \in \calL(E)$.
	Then the adjoint semigroup $\cT' \coloneqq (T_t')_{t \in J}$ converges with respect to the operator norm and thus also strongly to $P' \in \calL(E)$. 
	For every $0 < \psi \in E'$ we have 
	$\langle P'\psi, f_0 \rangle = \langle \psi, P f_0 \rangle = \langle \psi, f_0 \rangle > 0$ since $f_0$ is a quasi-interior point of $E_+$,
	which shows that $P' \psi > 0$.
	Proposition~\ref{prop:semigroup-of-lattice-homomorphisms-identity-operator} now implies that $T_t' = \id_{E'}$ for all $t \in J$. This completes the proof.
\end{proof}

As for instance observed by Lasota in \cite[p.\,398]{lasota1983} convergence in operator norm is rather uncommon for semigroups 
of Frobenius-Perron operators. The following corollary shows that it can almost never occur if the underlying transformations are 
measure preserving mappings on a finite measure space.

\begin{corollary} \label{cor:norm-convergence-frobenius-perron-semigroups}
	Let $(\Omega,\mu)$ be a finite measure space and let $J = \bbN$ or $J = (0,\infty)$. 
	For every $t \in J$ let $\varphi_t\colon \Omega \to \Omega$ be a measurable mapping which is measure-preserving, meaning that
	$\mu(\phi_t^{-1}(A)) = \mu(A)$ for all measurable sets $A \subseteq \Omega$, and assume that $\varphi_{t+s} = \varphi_t \circ \varphi_s$ for all $t,s \in J$.
	
	For each $t\in J$ denote by $T_t\colon L^1(\Omega,\mu) \to L^1(\Omega,\mu)$ the Frobenius-Perron operator associated with $\varphi_t$.
	If the semigroup $\cT = (T_t)_{t \in J}$ is convergent with respect to the operator norm, then $T_t = \id_{L^1(\Omega,\mu)}$ for all $t \in J$.
\end{corollary}
\begin{proof}
	For each $t \in J$ the adjoint operator $T_t'\colon L^\infty(\Omega,\mu) \to L^\infty(\Omega,\mu)$ is the Koopman operator associated with $\varphi_t$ 
	and hence a lattice homorphism. Moreover, since every mapping $\varphi_t$ is measure preserving and the measure space $(\Omega,\mu)$ is finite, 
	the constant function with value $1$ is a fixed vector of $\cT$. 
	Hence, the assertion follows from Corollary~\ref{cor:norm-convergence-of-sg-with-lattice-homomorphisms-as-adjoints}.
\end{proof}

The assertion of the above corollary does no longer hold if one drops the condition on $\varphi_t$ to be measure preserving. 
As a counterexample, let $\Omega = \{1,2\}$ be endowed with the counting measure, 
let $\varphi\colon \Omega \to \Omega$ be given by $\varphi(\omega) = 1$ for all $\omega \in \Omega$ 
and define $\varphi_t \coloneqq \varphi$ for all $t \in J$ (no matter whether $J = \bbN$ or $J = (0,\infty)$). 
Then we clearly have $\varphi_{t+s} = \varphi_t \circ \varphi_s$ for all $t,s \in J$ and the semigroup of the induced Frobenius-Perron operators is constant and 
therefore convergent with respect to the operator norm. Yet, the semigroup operators do not equal the identity operator.

\section{Lower bounds on Banach lattices} \label{section:lower-bounds-general-banach-lattices}

In the following we generalise some of the results of the previous sections from AL-spaces to more general Banach lattices. 
Theorem~\ref{thm:lower-bounds-only-on-al-spaces} however suggests that the previous concept of lower bounds 
is not well adapted to other Banach lattices. 
In a series of papers \cite{rudnicki1986, lasota1988, socala1998, socala2002, socala2003} Rudnicki, Lasota and Soca{\l}a 
used a modified lower bound condition by requiring, in some sense, 
the existence of a lower bound for the sequences $\bigl(\frac{T^nf}{\norm{T^nf}}\bigr)_{n \in \bbN}$ instead for the sequences $(T^nf)_{n \in \bbN}$,
in order to obtain convergence results on more general spaces.
Here, we use another approach and consider lower bounds \emph{with respect to a given strictly positive functional}.
A related idea can also be found in \cite{erkursun2016}. We make this precise in the following definition:

\begin{definition}
	Let $E$ be a Banach lattice and let $\cT = (T_t)_{t \in J}$ be a positive semigroup on $E$
	where either $J = \bbN$ or $J = (0,\infty)$. Let $\psi \in E'_+$. 
	\begin{enumerate}[(a)]
		\item A vector $h \in E_+$ is called a \emph{lower bound for $\cT$ with respect to $\psi$} if $\applied{\psi}{(T_tf-h)^-}\to 0$ as $t\to\infty$ for every
		$f \in E_+$ with $\langle \psi, f \rangle = 1$.
		\item Let $f \in E_+$. A vector $h \in E_+$ is called a \emph{lower bound for $(\cT,f)$ with respect to $\psi$} if $\langle \psi, (T_tf - h)^-\rangle \to 0$ as $t \to \infty$.
	\end{enumerate}
\end{definition}

Now we can prove a result similar to Theorem~\ref{thm:ind-lower-bounds-with-norm-bounds-for-markov-sg} on general Banach lattices with order continuous norm under an additional assumption on the semigroup.

\begin{theorem}
	\label{thm:ind-lower-bound-wrt-psi}
	Let $E$ be a Banach lattice with order continuous norm and let $\cT = (T_t)_{t \in J}$ be a bounded, positive 
	semigroup on $E$ where either $J = \bbN$ or $J =(0,\infty)$. 
	Assume that there exists a constant $M > 0$, a quasi-interior point $f_0$ of $E_+$ and a strictly positive functional $\psi \in E'_+$ 
	such that $T_t f_0 \le M f_0$ and $T_t' \psi \le M \psi$ for all $t \in J$. Then the following assertions are equivalent:
	\begin{enumerate}[(i)]
		\item $\cT$ is strongly convergent and the limit operator $P$ fulfils $P'\psi \ge \varepsilon \psi$ for some $\varepsilon > 0$.
		\item For every $f \in E_+$ with $\langle \psi, f \rangle = 1$ there exists a non-zero lower bound $h_f$ for $(\cT,f)$ with respect to $\psi$ such that $\inf_{f} \langle \psi, h_f \rangle > 0$.
	\end{enumerate}
\end{theorem}

Note that the inequalities $T_t f_0 \le M f_0$ and $T_t' \psi \le M \psi$ for all $t \in J$ and some constant $M > 0$ are in particular fulfilled 
if $f_0$ is a fixed vector and $\psi$ is a fixed functional of $\cT$; in this case one can, of course, choose $M = 1$.

\begin{proof}[Proof of Theorem~\ref{thm:ind-lower-bound-wrt-psi}]
	(i) $\Rightarrow$ (ii): If (i) holds, then for every $f \in E_+$ with $\langle \psi, f\rangle = 1$, $Pf$ is a lower bound for $(\cT,f)$ with respect to $\psi$ and we have $\langle \psi, Pf \rangle \ge \varepsilon > 0$.
	
	(ii) $\Rightarrow$ (i): Consider the norm $\norm{f}_\psi \coloneqq \applied{\psi}{\abs{f}}$ on $E$, which is clearly additive on the positive cone.
	Since $\norm{T_tf}_\psi \leq \applied{\psi}{T_t \abs{f}} \leq M \norm{f}_\psi$, the semigroup is bounded by $M$ with respect to $\norm{\argument}_\psi$.
	Denote by $F$ the completion of $(E,\norm{\argument}_\psi)$ and by $\cS = (S_t)_{t\in J}$ the extension of $\cT$ to $F$. Then $\cS$ is a positive semigroup on the AL-space $F$
	that is bounded by $M$.
	By assumption (ii) there exists a number $\delta > 0$ with the property that
	for every $f \in E_+$ we can find a vector $h_f \in E_+$ with $\norm{h_f}_\psi \ge \delta\norm{f}_\psi$ such that $h_f$ is a lower bound for $(\cS,f)$ in $F$;
	moreover, since
	\[ 0 \leq h_f = h_f - T_tf + T_tf \leq (T_tf-h_f)^- + T_tf \]
	for all $t\in J$, we also have that $\norm{h_f}_\psi \leq M \norm{f}_\psi$.

	Now, let $f \in F_+$ with $\norm{f}_\psi = 1$. We show that there exists a lower bound $h_f \in E_+$ for $(\cS,f)$ with $\norm{h_f} \ge \delta$.
	To this end, let $A \coloneqq \{g \in E_+:  g \le f\}$. Clearly, $A$ is an upwards directed set. 
	Since $E$ is dense in $F$ we can find a sequence $(f_n) \subseteq E$ which converges to $f$ with respect to $\norm{\argument}_\psi$. 
	As $E$ has order continuous norm, it is an ideal in $F$ \cite[Lem IV.9.3]{schaefer1974}. Hence, $(f_n^+ \land f)$ 
	is a sequence in $A$ which converges to $f$ with respect to $\norm{\argument}_\psi$ and hence the net $(g)_{g \in A}$ converges to $f$ with respect to $\norm{\argument}_\psi$. 
	By the same argument as in the proof of Lemma~\ref{lem:hmax-uniform} it follows that for every $g \in A\subseteq E_+$ 
	the set of all lower bounds for $(\cS,g)$ has a maximum $h_g \in F$ and this vector $h_g$ clearly 
	fulfils $\delta \norm{g}_\psi \le \norm{h_g}_\psi \le M\norm{g}_\psi$.
	Moreover, we clearly have $h_{g_1} \le h_{g_2}$ whenever $g_1 \le g_2$. Thus, the net $(h_g)_{g \in A}$ is increasing and norm bounded in the KB-space $F$ and hence convergent to a vector $h_f \in F$
	that fulfils
	\[ \norm{h_f}_\psi = \sup_{g\in A} \norm{h_g}_\psi \geq \sup_{g\in A} \delta\norm{g}_\psi = \delta \norm{f}_\psi = \delta.\]
	Moreover, $h_f$ is a lower bound for $(\cS,f)$. Indeed, for a given $\eps>0$ choose $g\in A$ such that $\norm{g-f}<\eps$ and $\norm{h_g-h_f}<\eps$. Then
	\[ \norm{(S_t f - h_f)^-}_\psi \leq \norm{S_t f-S_t g}_\psi + \norm{(S_tg-h_g)^-}_\psi + \norm{h_g-h_f}_\psi \leq M\eps +\eps+\eps\]
	whenever $\norm{(S_tg-h_g)^-}_\psi<\eps$. 
	Thus, the assumptions of Corollary~\ref{cor:ind-lower-bounds-with-norm-bounds-for-bounded-sg}(ii) are fulfilled and we conclude that $\cS$ is strongly convergent.
	
	Now, let $f$ be an element of the principal ideal $E_{f_0}$ and choose $c>0$ such that $\abs{f} \le cf_0$. We proved that $(T_t f)_{t\in J} = (S_tf)_{t\in J}$ 
	converges with respect to $\norm{\argument}_\psi$ to a vector $g_f \in F$ as $t \to \infty$. 
	For every $t \in J$ we have $\abs{T_tf} \le T_t\abs{f} \le cT_tf_0 \le cM f_0$ and therefore $\abs{g_f} \le c M f_0$. Since $E$ is an ideal in $F$, this shows that $g_f \in E$. As $E$ and $F$ have order continuous norm, the fact that $E$ is an ideal in $F$ moreover implies that the topologies of $\norm{\argument}$ and $\norm{\argument}_\psi$ coincide on
	every order interval in $E$ \cite[Thm 2.4.8]{meyer1991}.
	Therefore, the net $(T_tf)_{t \in J}$ converges to $g_f$ with respect to $\norm{\argument}$.
	
	We have shown that $\lim T_tf$ exists in $E$ for every $f$ in the principal ideal generated by $f_0$. 
	Since $f_0$ is a quasi-interior point, this principal ideal is dense in $E$ and as the semigroup $\cT$ is bounded, it follows that $(T_t)_{t \in J}$ converges strongly on $E$ to an operator $P\in \cL(E)$.
	Finally, let $\varepsilon \coloneqq \inf_{f} \langle \psi, h_f \rangle > 0$, where the infimum runs over all $f \in E_+$ with $\langle \psi, f \rangle = 1$. Then $P$ clearly fulfils $\langle \psi, Pf\rangle \ge \varepsilon \langle \psi, f \rangle$ for every $f \in E_+$.
\end{proof}

We also obtain a similar result for strong convergence to a rank-$1$ projection:

\begin{theorem}
	\label{thm:unif-lower-bound-wrt-psi}
	Let $E$ be a Banach lattice with order continuous norm and let $\cT = (T_t)_{t \in J}$ be a bounded positive semigroup on $E$ where either $J = \bbN$ or $J =(0,\infty)$. 
	Assume that there exists a constant $M > 0$, a quasi-interior point $f_0$ of $E_+$ and a strictly positive functional $\psi \in E'_+$ 
	such that $T_t f_0 \le M f_0$ and $T_t' \psi \le M \psi$ for all $t \in J$. Then the following assertions are equivalent:
	\begin{enumerate}[(i)]
		\item $\cT$ is strongly convergent and the limit operator $P$ is a rank-$1$ projection which fulfils $P'\psi \ge \varepsilon \psi$ for some $\varepsilon > 0$.
		\item There exists a non-zero lower bound $h$ for $\cT$ with respect to $\psi$.
	\end{enumerate}
\end{theorem}
\begin{proof}
	This follows from Corollary~\ref{cor:lasota-yorke-for-bounded-semigroups} by the same arguments we used to derive Theorem~\ref{thm:ind-lower-bound-wrt-psi} from 
	Corollary~\ref{cor:ind-lower-bounds-with-norm-bounds-for-bounded-sg}. We thus omit the details.
\end{proof}

Note that the above Theorems~\ref{thm:ind-lower-bound-wrt-psi} and~\ref{thm:unif-lower-bound-wrt-psi} yield something new even on AL-spaces, 
since we might encounter a semigroup $\cT$ on an AL-space which does not have a non-zero lower bound (with respect to the norm) but a non-zero lower bound with respect to some strictly positive functional.

\begin{remarks}
	(a) Using the methods from the proof of Theorem~\ref{thm:ind-lower-bound-wrt-psi} it is easy to prove a version of Theorem~\ref{thm:ding-abstract-and-for-bounded-sg} on Banach lattices 
	with order continuous norm, provided that there are a quasi-interior point $f_0 \in E_+$ and a strictly positive functional $\psi \in E'_+$ which fulfils similar properties as in Theorem~\ref{thm:ind-lower-bound-wrt-psi}: using the notation from the proof of Theorem~\ref{thm:ind-lower-bound-wrt-psi}, 
	one can use the fact that $E$ is an ideal in $F$ together with Proposition~\ref{prop:interval-preserving-if-the-adjoint-is-a-lattice-homomorphism} and \cite[Thm~1.4.19(ii)]{meyer1991} to show that the adjoint operators $S_t' \in \calL(F')$ 
	of the extended semigroup $\cS \subseteq \calL(F)$ are lattice homomorphisms and thus one can apply Theorem~\ref{thm:ding-abstract-and-for-bounded-sg}.
	
	(b) Similarly, one can prove a version of Theorem~\ref{thm:ind-lower-bounds-for-sg-of-lattice-homomorphisms}  under appropriate assumptions
	on Banach lattices with order continuous norm.
		
	(c) One cannot simply combine Theorem~\ref{thm:unif-lower-bound-wrt-psi} with an ultra power argument to obtain a result on convergence in operator 
	norm as we did in Corollary~\ref{cor:lasota-yorke-for-bounded-semigroups-and-convergence-in-operator-norm}. 
	Even if an ultra power $E^\cU$ of $E$ has again order continuous norm (which is e.g.\ true for any $L^p$-space where $p \in [1,\infty)$), 
	the following problems occur: the lifting $f_0^\cU \in (E^\cU)_+$ of a quasi-interior point $f_0 \in E_+$ is in general not a quasi-interior point of $(E^\cU)_+$; similarly, the lifting of a strictly positive functional $\varphi \in E'_+$ to the functional $\varphi^\cU \in (E^\cU)'$, given by $\langle \varphi^\cU, (f_n)^\cU \rangle = \lim_\cU \langle \varphi, f_n \rangle$ for all $(f_n)^\cU \in E^\cU$, is in general not a strictly positive functional on $E^\cU$.
\end{remarks}

The role of $\psi$ in Theorems~\ref{thm:ind-lower-bound-wrt-psi} and~\ref{thm:unif-lower-bound-wrt-psi} is quiet different from the role of $f_0$: 
the strictly positive functional $\psi$ is needed for the notion of a lower bound to make sense at all (see the discussion at the beginning of this section).
Hence it can be seen as some kind of substitute for the norm on AL-spaces and it appears in the equivalent conditions (i) and (ii) of the theorems. 
The quasi-interior point $f_0$ on the other hand does not appear in the conditions (i) and (ii) and it has no counterpart 
in Corollaries~\ref{cor:lasota-yorke-for-bounded-semigroups} and~\ref{cor:ind-lower-bounds-with-norm-bounds-for-bounded-sg} on AL-spaces; 
its only purpose is to ensure that the orbits of the semigroup are order bounded in $E$. 
Hence, the existence of $f_0$ is an additional assumption which was superfluous on AL-spaces and it is natural to ask whether this condition 
can be omitted in Theorems~\ref{thm:ind-lower-bound-wrt-psi} and~\ref{thm:unif-lower-bound-wrt-psi}. Our answer to this question consists of two parts: 
first we give a counterexample which shows that 
condition is needed in general, see Example~\ref{ex:counter-example-quasi-interior-fixed-point}. 
Afterwards we show that the condition can, however, be dropped if the orbits of the semigroup are relatively compact,
see Propositions~\ref{prop:no-fixed-point-but-compact-orbits} and~\ref{prop:no-fixed-point-but-compact-orbits-without-os-norm} below.

\begin{example} \label{ex:counter-example-quasi-interior-fixed-point}
	Fix $p \in (1,\infty)$. There exists a finite measure $\mu$ on $\bbN_0$, a bounded linear operator $T$ on $\ell^p \coloneqq \ell^p(\bbN_0,\mu)$ 
	and a strictly positive linear functional $\psi$ on $\ell^p$ with the following properties:
	\begin{enumerate}[(a)]
		\item $T$ is positive and power-bounded.
		\item $T'\psi = \psi$ and there exists a non-zero lower bound $\cT = (T^n)_{n \in \bbN}$ with respect to $\psi$.
		\item The semigroup $(T^n)_{n\in\N}$ is not strongly convergent.
	\end{enumerate}
	Indeed, define $\mu$ by $\mu(M) = \sum_{k\in M} \frac{1}{k^p}$ for all $M \subseteq \bbN$ and $\mu(\{0\}) = 1$. Since $p > 1$, this measure is finite. 
	We first define an operator $T_1$ on $\ell^1 \coloneqq \ell^1(\bbN_0,\mu)$ and a functional $\psi_1$ on $\ell^1$ and 
	then we obtain $T$ and $\psi$ as their restrictions to $\ell^p$. 
	Let $\psi_1$ be the norm functional on $\ell^1$, i.e.\ 
	\begin{align*}
		\langle \psi_1, f \rangle = \sum_{k \in \bbN_0} f(k) \mu(\{k\}) = f(0) + \sum_{k=1}^\infty f(k) \frac{1}{k^p}
	\end{align*}
	$f \in \ell^1$. Note that $\ell^p(\bbN_0,\mu) \eqqcolon \ell^p \subseteq \ell^1$ 
	since $\mu$ is a finite measure and let $\psi \coloneqq \psi_1|_{\ell^p}$. Clearly, $\psi$ is strictly positive.
	To define the operator $T_1$ we introduce another positive functional $\alpha_1$ on $\ell^1$, given by
	\begin{align*}
		\langle \alpha_1,f \rangle = \sum_{k=2}^\infty \frac{k}{k-1}f(k-1) \frac{1}{k^p} = \sum_{k=1}^\infty \frac{k+1}{k} f(k) \frac{1}{(k+1)^p}
	\end{align*}
	for each $f \in \ell^1$. One easily checks that $\langle \alpha_1, f\rangle \le \langle \psi_1, f \rangle = \norm{f}_1$ for all $f \in \ell^1_+$, 
	so $\alpha_1$ is dominated by $\psi_1$ and contractive.
	Now we define the operator $T_1 \in \calL(\ell^1)$ by
	\begin{align*}
		(T_1f)(k) =
		\begin{cases}
			\langle \psi_1-\alpha_1, f\rangle \quad & \text{if } k = 0 \\
			0 \quad & \text{if } k = 1 \\
			\frac{k}{k-1} f(k-1) \quad & \text{if } k \ge 2
		\end{cases}
	\end{align*}
	This is easily checked to be indeed a positive linear operator on $\ell^1$. The functional $\psi_1$ is a fixed point of $T_1'$ since we have 
	\begin{align*}
		\langle \psi_1, Tf \rangle = \langle \psi_1, (Tf)|_{\{0\}} + (Tf)|_{\bbN}\rangle & = (Tf)(0) + \sum_{k=2}^\infty \frac{k}{k-1}f(k-1) \frac{1}{k^p}  \\
		& = \langle \psi_1 - \alpha_1, f \rangle + \langle \alpha_1, f\rangle = \langle \psi_1, f\rangle,
	\end{align*}
	for every $f \in \ell^1$; in the above computation, we consider $(Tf)|_{\{0\}}$ and $(Tf)|_{\bbN}$ again as elements of $\ell^1$ which are zero on $\bbN$ and $\{0\}$, respectively. As $T_1' \psi_1 = \psi_1$, we conclude that $T_1$ is a Markov operator. If $f \in \ell^p$, then we have
	\begin{align*}
		\norm{T_1f}_p^p & = \sum_{k=0}^\infty \abs{Tf(k)}^p \mu(\{k\}) = \abs{\langle\psi_1-\alpha_1,f\rangle}^p + \sum_{k=2}^\infty \frac{k^p}{(k-1)^p} \abs{f(k-1)}^p \frac{1}{k^p} \\
		& \le \langle \psi, \abs{f}\rangle^p + \sum_{k=1}^\infty \abs{f(k)}^p \frac{1}{k^p} \le (\norm{\psi}^p+1) \norm{f}_p^p < \infty.
	\end{align*}
	Hence, $T_1$ leaves $\ell^p$ invariant and its restriction $T\coloneqq T_1|_{\ell^p}$ is a bounded operator on $\ell^p$.
	
	In order to prove the claimed properties (a)--(c) of $T$, we first compute all powers of $T_1$ applied to the canonical unit vectors $e_j$. 
	Of course we have $T_1e_0 = e_0$ and hence $T_1^ne_0 = e_0$ for all $n \in \N_0$. For $j\in\N$ and $n\in\N_0$ we have
	\begin{align}
		\label{eqn:Tnej}
		(T_1^ne_j)|_\bbN = \frac{j+n}{j} \cdot e_{j+n}.
	\end{align}
	Now we can prove (a)--(c):
	
	(a) Since $T_1$ is positive, so is $T$. To prove that $T$ is power-bounded it is, due to the Uniform Boundedness Principle,  
	sufficient to prove that $(T^nf)_{n \in \bbN_0}$ is bounded in $\ell^p$ for each $0 \le f \in \ell^p$. So let $f$ be such a vector; 
	for every $n \ge 1$ we have $0 \le (T^nf)(0) = (T_1^nf)(0) = (T_1T_1^{n-1}f)(0) \le \langle \psi_1, T_1^{n-1}f\rangle = \langle \psi_1, f \rangle$. 
	Therefore, it suffices to show that the sequence $((T^nf)|_{\bbN})_{n \in \bbN_0}$ is bounded in $\ell^p$. This follows from the computation
	\begin{align*}
		\norm{(T^nf)|_{\bbN}}_p^p & = \norm{(T^n (f|_{\bbN}))|_\bbN}_p^p = \norm[\Bigg]{\biggl(T^n \sum_{j=1}^\infty f(j)e_j\biggr)|_\bbN}_p^p 
		= \norm[\Bigg]{\sum_{j=1}^\infty \frac{j+n}{j}f(j)e_{j+n}}_p^p  \\
		& = \sum_{j=1}^\infty \frac{(j+n)^p}{j^p} f(j)^p \frac{1}{(j+n)^p} = \norm{f|_{\bbN}}_p^p.
	\end{align*}
	This completes the proof of (a).
	
	(c) Note that the above computation also shows that $\norm{(T^nf)|_\bbN}_p = \norm{f|_\bbN}_p$ for each $n \in \N_0$ and each $0 \le f \in \ell^p$. 
	However, as we have seen in \eqref{eqn:Tnej}, $T^ne_1$ converges pointwise to $0$ and can therefore not be convergent in $\ell^p$.
	
	(b) Since we have $T_1'\psi_1 = \psi_1$, it follows that $T\psi = \psi$. To prove that $T$ has a non-zero lower bound with respect to $\psi$, we first 
	show that $T_1^n$  converges strongly as $n \to \infty$. Let $f \in \ell^1_+$. Then
	\begin{align*}
		\norm{(T_1^nf)|_{\bbN}}_1 & = \norm[\Bigg]{\biggl(T^n\sum_{j=1}^\infty f(j)e_j\biggr)|_\bbN}_1 = \norm[\Bigg]{\sum_{j=1}^\infty \frac{j+n}{j}f(j)e_{j+n}}_1  \\
		& = \sum_{j=1}^\infty \frac{j+n}{j} f(j) \frac{1}{(j+n)^p} = \sum_{j=1}^\infty \frac{j^{p-1}}{(j+n)^{p-1}} f(j) \frac{1}{j^p} \to 0 
	\end{align*}
	as $n \to \infty$; the convergence follows from the dominated convergence theorem since $p > 1$ and $\sum_{j=1}^\infty f(j) \frac{1}{j^p} < \infty$. 
	Moreover, 
	\begin{align*}
		(T_1^nf)(0) = (T_1T_1^{n-1}f)(0) & = \langle \psi_1 - \alpha_1, T_1^{n-1}f\rangle \\
		& = \langle \psi_1, f\rangle - \langle \alpha_1, (T_1^{n-1}f)|_{\bbN} \rangle \to \langle \psi_1, f \rangle \quad \text{as } n \to \infty.
	\end{align*}
	Hence $T_1^nf \to \langle \psi_1, f \rangle e_0$ and by linearity it follows that $(T_1^n)$ converges strongly to $\psi_1\otimes e_0$.
	This implies that $e_0$ is a lower bound for $(T_1^n)_{n \in \bbN}$ and thus it is also a lower bound for $(T^n)_{n \in \bbN}$ with respect to $\psi$.
\end{example}

Recall that a semigroup $\cT = (T_t)_{t \in J}$ on a Banach space $E$ is said to \emph{have relatively compact orbits} if for each $f\in E$ the set $\{T_tf: t \in J\}$ is 
relatively compact with respect to the norm on $E$. This condition is frequently satisfied in applications. 
For example, if $\cT = (T_t)_{t \in [0,\infty)}$ is a bounded $C_0$-semigroup whose generator has compact resolvent, then $\cT$ has relatively compact orbits \cite[Cor V.2.15]{nagel2000}.

\begin{proposition} \label{prop:no-fixed-point-but-compact-orbits}
	Let $E$ be a Banach lattice with order continuous norm and let $\cT = (T_t)_{t \in J}$ be a positive semigroup on a $E$ where either $J = \bbN$ or $J = (0,\infty)$. 
	Suppose that there exists a strictly positive functional $\psi \in E'_+$ such that $T'_t\psi \le M \psi$ for a constant $M>0$ and all $t \in J$. 
	
	Assume that $\cT$ has relatively compact orbits and suppose that for every $f \in E_+$ with $\langle \psi, f \rangle = 1$ 
	there exists a non-zero lower bound $h_f$ for $(\cT,f)$ with respect to $\psi$  
	such that $\inf_{f} \langle \psi, h_f \rangle \ge 1$. Then $\cT$ is strongly convergent.
\end{proposition}
\begin{proof}
	We repeat the first steps of the proof of Theorem \ref{thm:ind-lower-bound-wrt-psi}: The semigroup $\cT$ is bounded with respect to the norm
	$\norm{f}_\psi \coloneqq \applied{\psi}{\abs{f}}$, which is additive on the positive cone $E_+$. 
	Denote by $F$ the completion of $(E,\norm{\argument}_\psi)$ and by $\cS = (S_t)_{t\in J}$ the extension of $\cT$ to $F$. 
	Using that $E$ has order continuous norm, one can show as in the proof of Theorem~\ref{thm:ind-lower-bound-wrt-psi} that the assumptions 
	of Corollary~\ref{cor:ind-lower-bounds-with-norm-bounds-for-bounded-sg}(ii) are fulfilled and thus, $(S_t)$ converges strongly to an operator $P\in \cL(F)$ as $t \to \infty$.

	Now pick $0<f\in E$ and consider an arbitrary sequence of times $(t_n) \subseteq J$ which converges to $\infty$. Since the orbit $\{T_t f: t \in J\}$ is relatively compact in $E$, 
	we find a subsequence $(t_{n_k})$ of $(t_n)$ such that $(T_{t_{n_k}}f)$ converges to a point $g \in E$ with respect to the norm in $E$. 
	In particular, $(T_{t_{n_k}}f)$ converges to $g$ with respect to $\norm{\argument}_\psi$ and thus, we have $g = Pf$. This shows that $Pf \in E$ and that 
	every subsequence of $(T_tf)_{t \in J}$ has a subsequence which converges to $Pf$. Hence, $\lim T_tf = Pf$ in $E$.
\end{proof}

In the above proposition, order continuity of the norm was only needed to show that the extended semigroup $\cS$ also admits individual lower bounds which are bounded below in norm. If we have a single lower bound for the entire semigroup, this argument is much simpler and thus, we do not need order continuity of the norm:

\begin{proposition} \label{prop:no-fixed-point-but-compact-orbits-without-os-norm}
	Let $E$ be a Banach lattice and let $\cT = (T_t)_{t \in J}$ be a positive semigroup on $E$ where either $J = \bbN$ or $J = (0,\infty)$. 
	Assume that there exists a strictly positive functional $\psi \in E'_+$ such that $T'_t\psi \le M \psi$ for a constant $M>0$ and all $t \in J$. 
	If $\cT$ has relatively compact orbits and a non-zero lower bound with respect to $\psi$, then $\cT$ converges strongly to a rank-$1$ projection.
\end{proposition}

The proof is almost the same as for Proposition~\ref{prop:no-fixed-point-but-compact-orbits}, except for the fact that 
one refers to Corollary~\ref{cor:lasota-yorke-for-bounded-semigroups} instead of Corollary~\ref{cor:ind-lower-bounds-with-norm-bounds-for-bounded-sg}.
It is easy to see that the non-zero lower bound for the original semigroup is also a lower bound for the extended one and thus,
one does not need order continuity of the norm here.

\subsection*{Acknowledgements} The second author is indebted to Marta Tyran-Kami\'{n}ska and Ryszard Rudnicki for a very interesting discussion at the workshop ``Evolution equations: theory and applications'' in spring 2015 in Besan\c{c}on, as well as for bringing Ding's paper \cite{ding2003} to his attention. He was also supported by a scholarship within the scope of the LGFG Baden-W\"urttemberg (grant number 1301 LGFG-E), Germany.

\bibliographystyle{abbrv}
\bibliography{literature}

\end{document}